\documentclass[11pt, reqno]{article}
\pdfoutput=1

\usepackage{graphicx}
\usepackage{jheppub}
\usepackage{amssymb}
\usepackage{amsmath}
\usepackage[usenames,dvipsnames]{xcolor}
\usepackage{epsfig}
\usepackage{dcolumn}
\usepackage{tikz}
\usetikzlibrary{shapes.geometric, arrows}
\usepackage{upgreek}
\usepackage{setspace}
\usepackage{enumitem}
\usepackage{array,multirow,bigdelim,arydshln}
\usepackage{appendix}
\usepackage{xparse}
\usepackage[utf8]{inputenc}
\usepackage{hyperref}
\hypersetup{
	colorlinks,
	urlcolor=Maroon,
	linkcolor=Maroon,
	citecolor=Maroon
}

\usepackage{amsthm}
\newtheorem{thm}{Theorem}[section]
\newtheorem{prop}[thm]{Proposition}

\newtheorem{lem}[thm]{Lemma}
\newtheorem{cor}[thm]{Corollary}
\newtheorem{defn}[thm]{Definition}
\theoremstyle{definition}
\newtheorem{example}[thm]{Example}

\newtheorem{conjecture}[thm]{Conjecture}
\newtheorem{claim}[thm]{Claim}
\newtheorem{rem}[thm]{Remark}

\usepackage{mathtools}

\usepackage{float}
\restylefloat{table}

\NewDocumentCommand{\binomial}{omm}
{%
	\genfrac(){0pt}{}{#2}{#3}%
	\IfValueT{#1}{_{\!#1}}%
}
\NewDocumentCommand{\eulerian}{omm}
{%
	\genfrac<>{0pt}{}{#2}{#3}%
	\IfValueT{#1}{_{\!#1}}%
}

\def \s {\sigma}

\usepackage{underscore}

\usepackage{latexsym}
\usepackage{tikz}

\def \be {\begin{equation}}
	\def \ee {\end{equation}}
\def \ba {\begin{eqnarray}}
	\def \ea {\end{eqnarray}}

\def \be {\begin{equation}}
	\def \en {\end{equation}}
\def \bes {\begin{eqnarray}}
	\def \ens {\end{eqnarray}}

\def \s {\textsf{s}}

\def \C {\textsf{C}}
\def \tp {||}

\def \x {z}
\def \V {\textsf{V}}

\def \w {\omega}
\def \e {\epsilon}
\def \spk {\mathcal{S}^{(PK)}_{k,n}}
\newcommand{\cat}[2]{C^{(#1)}_{#2}}

\numberwithin{equation}{section}

\title{Planar Kinematics: Cyclic Fixed Points, Mirror Superpotential, k-Dimensional Catalan Numbers, and Root Polytopes}

\author{Freddy Cachazo and}\emailAdd{fcachazo@pitp.ca}
\author{Nick Early}\emailAdd{earlnick@gmail.com}

\affiliation{Perimeter Institute for Theoretical Physics, Waterloo, ON N2L 2Y5, Canada}

\abstract{
	In this paper we prove that points in the space  $X(k,n)$ of configurations of $n$ points in $\mathbb{CP}^{k-1}$ which are fixed under a certain cyclic action are the solutions to the generalized scattering equations on planar kinematics (PK). In the first part, we give a constructive upper bound: we show that these solutions inject into certain aperiodic k-element subsets of $\{1,\ldots, n\}$, and consequently that their number is bounded above by the number of Lyndon words with k one's and n-k zeros. The proof uses a somewhat surprising connection between the superpotential of the mirror of $G(n-k,n)$ and the generalized CHY potential on $X(k,n)$. We also check the recent conjecture that generalized biadjoint amplitudes evaluate to $k$-dimensional Catalan numbers on PK for several examples including $k=3$ and $n\leq 40$ and $(k,n)=(6,13)$. We then reformulate the CEGM generalized biadjoint scalar amplitude directly as a Laplace transform-type integral over ${\rm Trop}^+ G(k,n)$ and we use it to evaluate the amplitude on PK with the purpose of exhibiting how Generalized Feynman Diagrams glue together. 
	
	We initiate the study of two minimal lattice polytopal neighborhoods of the planar kinematics point.  One of these, the rank-graded root polytope $\mathcal{R}_{k,n}$, in the case $k=2$, is a projection of the standard type A root polytope.  The other, denoted $\Pi_{k,n}$, in the case $k=2$, is a degeneration of the associahedron.   We check up to and including $\mathcal{R}_{3,9}$ and $\mathcal{R}_{4,9}$ that the relative volume of $\mathcal{R}_{k,n}$ is the multi-dimensional Catalan number $C^{(k)}_{n-k}$, hinting towards the possibility of deeper geometric and combinatorial interpretations of $m^{(k)}(\mathbb{I}_n,\mathbb{I}_n)$ near the PK point.

}

\begin{document}
	\maketitle
	% \addtocontents{toc}{\protect\setcounter{tocdepth}{1}}
	\addtocontents{toc}{\protect\setcounter{tocdepth}{2}}

	%%%%%%%%%%%%%%%%%%%%%%%%%%%%%%%%%%%%%
	
	\section{Introduction}\label{sec:introduction}
	
	%%%%%%%%%%%%%%%%%%%%%%%%%%%%%%%%%%%%%%
	
	Motivated by Cachazo-He-Yuan (CHY) definition of biadjoint double partial amplitudes, $m_n(\mathbb{I}, \mathbb{I})$, as integrals over the configuration space of $n$ points on $\mathbb{CP}^{1}$ localized to points satisfying the scattering equations \cite{Cachazo:2013gna,Cachazo:2013hca,Cachazo:2013iea,Fairlie:1972zz,Fairlie:2008dg}, Guevara, and Mizera and the two authors (CEGM) introduced a generalization of the CHY formulation that uses the configuration space of $n$ points on $\mathbb{CP}^{k-1}$ \cite{Cachazo:2019ngv,Cachazo:2019apa,Cachazo:2019ble} usually denoted by $X(k,n)$.   
	This also led to generalized biadjoint amplitudes $m_n^{(k)}(\mathbb{I}, \mathbb{I})$. 
	
	In 2013, CHY noticed that the kinematic invariants of all possible planar poles in a $k=2$ biadjoint amplitude form a basis of the corresponding kinematic space \cite{Cachazo:2013iea}. Using this fact CHY set all planar kinematic invariants to unity so that each planar Feynman diagram contributes exactly $1$ to the amplitude leading to the result that $m_n^{(2)}(\mathbb{I},\mathbb{I}) = C_{n-2}$ with $C_m$ the $m^{\rm th}$ Catalan number. 
	
	Very recently in \cite{MKandPK} the authors proposed a generalization of the $k=2$ {\it planar-basis} kinematics to all $k$ and $n$ using the planar basis introduced by the second author in \cite{Early:2019eun}.  This kinematics turns out to be a single integer point in the kinematic space which we call the \textit{PK point}.  In \cite{MKandPK} the scattering equations were solved for $k=3$ and $n=5,6,7,8$ and the corresponding CEGM biadjoint amplitudes were evaluated. The explicit results led the authors to conjecture that these amplitudes evaluate to the $k$-dimensional Catalan numbers (see O.E.I.S. A060854 \cite{oeis}), i.e. 
	\begin{eqnarray}\label{eq:planar kinematics amplitude intro}
		m_n^{(k)}(\mathbb{I}, \mathbb{I}) & = & \cat{k}{n-k}.
	\end{eqnarray}
	Clearly, $C^{(2)}_{n-2}$ coincides with the standard Catalan numbers.    
	
	In this note we continue the study of the scattering equations evaluated on the PK point, and its deformations: this culminates in Section \ref{sec: blades, polytopes} where we initiate the study of two minimal polytopal neighborhoods in the integer lattice in the kinematic space which are closely linked to the evaluation of the amplitude.  Here the PK point is the integer point in the kinematic space where a certain family of $\binom{n}{k}$ linear functions, denoted $\eta_J$ for $J$ a k-element subset of $\{1,\ldots, n\}$, on the kinematic space are either 0 or 1.  
	
	A linear functional $\eta_J$, introduced in Section \ref{sec: Blades, Planar Bases and Polymatroidal Blade Arrangements} is dual to piecewise linear surface over a hypersimplex, pinned to to one of its vertices such that the bends define a tropical hypersurface called a \textit{blade}.  In \cite{Early:2019zyi} it was shown that \textit{matroidal} blade arrangements on a hypersimplex $\Delta_{k,n}$ are in bijection with weakly separated collections\footnote{The weak separation condition was introduced by Leclerc and Zelevinsky in \cite{leclerc1998quasicommuting}.} and thus could be viewed as living inside the homogeneous component of a subalgebra of the cluster algebra of the Grassmannian $G(k,n)$.  
	
	Here the linear functions $\eta_J$ are constructed by lifting certain positroidal multi-split matroid subdivisions of the hypersimplex $\Delta_{k,n}$ to a piecewise linear surface and pairing with a point in the kinematic space.  See Equation \eqref{eq:planar basis element} for the definition and \cite{Early:2019eun} and \cite{Early:2020hap} for details and related constructions in combinatorics and applications to generalized Feynman diagrams.
	
	Specifically, $\eta_J=0$ for cyclically consecutive subsets $J = \{i,i+1,\ldots, i+(k-1)\}$ and $\eta_J=1$ for all of the remaining $k$-element subsets of $\{1,\ldots, n\}$.  Solving these $\binom{n}{k}$ equations gives the point in kinematic space which we call planar kinematics.
	
	The solution has the following simple formula.  Fixing a planar ordering such as the canonical order $\mathbb{I}:=(1,2,\ldots ,n-1,n)$, set
	\begin{align}\label{planarK}
		\s_{12\ldots k}=\s_{23\ldots k+1}=\ldots =\s_{n1\ldots k-1}=1, \\
		\s_{n1\ldots k-2,k}=\s_{12\ldots k-1,k+1}=\ldots = \s_{n-1,n\ldots k-3,k-1}=-1,\nonumber
	\end{align}
	where all other $\s_J$ are set to zero.
	
	The fact that the kinematics is cyclically invariant, i.e., invariant under a cyclic shift of the labels $i \to i+1\, \mod \, (n)$, motivated us to look for solutions with the same property. In other words, we are interested in points in $X(k,n)$ which are fixed under a cyclic shift. In Sections \ref{sec: critical points}, \ref{sec: superpotential PK potential equivalence} and \ref{sec:enumeration cyclic fixed points}, we find all such points and prove that they are indeed all the solutions to the scattering equations on planar kinematics. Some such points do not lie in $X(k,n)$ but in its compactification $\overline{X}(k,n)$.
	
	The proof, given in Section \ref{sec: superpotential PK potential equivalence}, uses that there is a very close relation between the scattering equations on planar kinematics, i.e. the equations for the critical points of 
	\be\label{introPot}
	\spk = \sum_{i=1}^n \log \left( \frac{\Delta_{i,i+1,\ldots i+(k-2),i+(k-1)}}{\Delta_{i,i+1,\ldots i+(k-2),i+k}} \right)
	\ee
	and those for the critical points of the superpotential in the theory mirror to the Grassmannian $G(n-k,n)$ introduced by Marsh and Rietsch in \cite{marsh2020b}
	\be\label{kapi}
	{\cal F}_q := \sum_{i=1,\, i\neq n-k}^{n}\frac{\Delta_{i,i+1,\ldots, i+(k-2),i+k}}{\Delta_{i,i+1,\ldots, i+(k-2),i+(k-1)}} + q \frac{\Delta_{n-k,n-k+1,\ldots,n-1,1}}{\Delta_{n-k,n-k+1,\ldots, n-1,n}}.
	\ee
	Here $\Delta_{i_1,i_2\ldots, i_k}$ are the Plucker coordinates of $G(k,n)$ and $q$ is a parameter.
	
	In \cite{karp2019moment} Karp proved that all critical points of the mirror superpotential, ${\cal F}_q$, are in fact fixed points under a cyclic action. Our proof provides the criteria for a fixed point in $G(k,n)$ to descend to one in $X(k,n)$ and become a critical point of $\spk$. In Section \ref{sec:enumeration cyclic fixed points}, we prove that these solutions inject into aperiodic k-element subsets of $\{1,\ldots, n\}$, and consequently that their number is bounded above by the number of Lyndon words with k one's and n-k zeros.
	
	The construction of the fixed points is explicit and therefore it is possible to evaluate the CEGM biadjoint amplitudes on them.  In Section \ref{sec: Evaluating CEGM Biadjoint Amplitudes}, we develop new techniques to evaluate the CEGM biadjoint amplitude, for k=3,4 in particular.
	
	In Section \ref{sec:CEGM Amplitudes on Planar Kinematics}, we perform the explicit evaluations up to $(k,n)=(3,40)$, $(k,n)=(4,29)$, $(k,n)=(5,19)$, and $(k,n)=(6,13)$. In all cases we find perfect agreement with the $k$-dimensional Catalan numbers. 
	
	Since the $2$-dimensional Catalan numbers count planar Feynman diagrams and CEGM biadjoint amplitudes have been related to the positive tropical Grassmannian ${\rm Trop}^+ G(k,n)$, it is natural to ask what the $k$-dimensional Catalan numbers are counting. The CEGM biadjoint amplitudes have also been defined as the sum over generalized Feynman diagrams (GFD) \cite{Borges:2019csl,Guevara:2020lek}. However, it is known that that for $k>2$ they are not counted by higher dimensional Catalan numbers \cite{MKandPK}. In fact, on planar kinematics individual GFD's evaluate to rational numbers. Motivated by this puzzle, in Section \ref{sec:tropGrass Eval} we introduce an integral of an exponentiated piecewise linear function supported on $\mathbb{R}^{(k-1)(n-k-1)}$ which computes the amplitude. The integral can be thought of as the Laplace transform of ${\rm Trop}^+ G(k,n)$.
	
	In fact, in the examples we studied, when the integral is evaluated on generic kinematics it splits into regions which coincide with individual GFD's. However, on planar kinematics it simplifies and the number of linear regions is much smaller. Moreover, each such region contributes a positive integer number hinting the existence of a polytopal interpretation.

	In Section \ref{sec: blades, polytopes}, we initiate the study of two families of lattice polytopes which are related by duality: first, we define rank-graded root polytopes $\hat{\mathcal{R}}_{k,n}$, and in particular, their projections, the root polytopes $\mathcal{R}_{k,n}$.  In the case $k=2$, then $\hat{\mathcal{R}}_{2,n}$ coincides with the usual root polytope introduced in \cite{postnikov2009permutohedra}, which is the convex hull of the origin together with the set of positive roots $e_i-e_j$ for $i<j$.  Moreover, $\mathcal{R}_{2,n}$ is a codimension 1 projection of it,

	Also in Section \ref{sec: blades, polytopes}, we initiate the study of a family of lattice polytopes $\Pi_{k,n}$ which are in duality with the polytopes $\mathcal{R}_{k,n}$ and which specialize in the case k=2 to a degeneration of the associahedron.  We show that $\Pi_{k,n}$ minimally bounds the PK point in the integer lattice in the kinematic space.  We conjecture the expression of $\Pi_{k,n}$ as a Newton polytope.
	
	Based on computations in SageMath of the volume of $\mathcal{R}_{k,n}$ for nontrivial values of $k$ and $n$, including $\mathcal{R}_{3,9}$ and $\mathcal{R}_{4,9}$, we finally conjecture that the rank-graded root polytope $\mathcal{R}_{k,n}$ has volume the multi-dimensional Catalan number $\cat{k}{n-k}$, thus hinting towards a deeper polytopal (and in particular combinatorial) interpretation of Equation \eqref{eq:planar kinematics amplitude intro}.

	%%%%%%%%%%%%%%%%%%%%%%%%%%%%%%%%%%%%%
	
	\section{Motivation: Fixed Points under a Cyclic Shift on $\overline{X}(k,n)$} \label{sec: cyclic shift}
	
	%%%%%%%%%%%%%%%%%%%%%%%%%%%%%%%%%%%%%%
	
	The space $X(k,n)$ of configurations of $n$ labeled points on $\mathbb{CP}^{k-1}$ can be represented by selecting homogeneous coordinates for the $n$ points and arranging them in a $k\times n$ matrix. The space can be formally defined as
	\be
	X(k,n) := SL(k)\backslash M^*(k,n) / \left(\mathbb{C}^*\right)^n
	\ee
	where $M^*(k,n)$ is the set of all $k\times n$ matrices with no vanishing minors. $SL(k)$ is the automorphism group of $\mathbb{CP}^{k-1}$ while the algebraic torus $\left(\mathbb{C}^*\right)^n$ corresponds to the projective action on each point.
	
	In order to study the solutions to the scattering equations in the next section, it turns out to be necessary to also include configurations of points represented by $k\times n$ matrices with vanishing minors whose only constraint is that no column is identically zero (so that the point is in $\mathbb{CP}^{k-1}$) and have maximal rank $k$, so that the action of $SL(k)$ is well-defined. We denote the extended space by $\overline{X}(k,n)$. A formal definition of $\overline{X}(k,n)$ as the compactification of $X(k,n)$ is beyond the scope of this work since we are only interested in particular points so a set-theoretic description suffices. 
	
	In order to find the points of interest we have to define a cyclic action on $\overline{X}(k,n)$.  Denote by $\mathbf{T} \hookrightarrow GL(n)$ the embedding of the torus $(\mathbb{C}^\ast)^n$ into the diagonal of $GL(n)$.
	
	With $e_1,\ldots, e_n$ the standard basis for $\mathbb{C}^n$, define a linear operator $\rho_n\in GL_n$ by $$\rho_n(e_j) = e_{j-1}$$
	where the indices are cyclic modulo $n$.
	
	The Grassmannian $G(k,n)$ admits a natural right-action of the cyclic group $\mathbb{Z}\slash n \simeq \langle \rho_n\rangle$ generated by $\rho_n$:
	$$\rho_n(g) = g\cdot \rho^{-1}_n.$$
	
	Further, letting $\mathbf{G}$ be the embedding of the semi-direct product $\mathbf{T}\rtimes \mathbb{Z}\slash n$ into $GL(n)$, then $\mathbf{G}$ acts on $\overline{X}(k,n)$ from the right; in a slight abuse of terminology, we will simply say that a torus orbit $\lbrack g\rbrack$ that is preserved by $\mathbf{G}$ is a \textit{cyclic fixed point}.  We are interested in the set of cyclic fixed points in $\overline{X}(k,n)$ of $\mathbf{G}$.
	
	In order to clarify the discussion which follows, let us be completely explicit about what it means for an element of $\overline{X}(k,n)$ to be fixed by $\mathbf{G}$.  Given $g\in G(k,n)$, denote by $\lbrack g\rbrack \in \overline{X}(k,n)$ the $T$-orbit of $g$.
	
	\begin{prop}\label{prop: cyclic}
		An element $g\in G(k,n)$ descends to a cyclic fixed point $\lbrack g\rbrack \in \overline{X}(k,n)$ provided that for any $\lambda_1\in\mathbf{T}$ we have
		$$(g\cdot \lambda_1)\rho_n^{-1} = g\cdot \lambda_2$$
		for some $\lambda_2\in \mathbf{T}$.
	\end{prop}
	
	We are interested in finding all fixed points of the cyclic shift. This is easily done by using $SL(k)$ and the torus action to fix the first column of the $k\times n$ matrix to be $(1,1,\ldots ,1)^T$. Consider any row of the matrix and denote its elements as $(1, x_{1},x_{2},\ldots ,x_{n-2},x_{n-1})$. The action of the cyclic shift is
	\be
	(1, x_{1},x_{2},\ldots ,x_{n-2},x_{n-1}) \longrightarrow (x_{1},x_{2},\ldots ,x_{n-2},x_{n-1},1).
	\ee
	Let us impose the condition that this be a fixed point. It is often convenient to combine the diagonal $GL(1)$ in $\left(\mathbb{C}^*\right)^n$ with $SL(k)$ into a $GL(k)$ action. In fact, in order to compare the matrix after the shift with the original one it is necessary to apply a $GL(k)$ transformation that multiplies the row by $1/x_{1}$ so as to normalize the first component. Having done this we have to require 
	\be
	(1, x_{1},x_{2},\ldots ,x_{n-2},x_{n-1}) = (1,x_{2}/x_{1},\ldots ,x_{n-2}/x_{1},x_{n-1}/x_{1},1/x_{1}).
	\ee
	These $n-1$ equations are equivalent to $x_i = (x_1)^i$ for $i\in \{2,3,\ldots ,n-1\}$ and $(x_1)^n=1$. Denoting $q=\exp 2\pi i/n$ the basic root of unity, one has $n$ possibilities for $x_1$ given by the $\{0^{\rm th},1^{\rm st},2^{\rm nd},\ldots ,(n-1)^{\rm th}\}$ powers of $q$.  
	
	It is now clear that in order to obtain a cyclic fixed point each row  in $k\times n$ matrix representative of the point must have the form
	\be
	\left( 1,\omega_a,\omega_a^2,\ldots ,\omega_a^{n-1} \right)
	\ee
	with $\omega_a := q^{m_a}$ and $m_a\in \{0,1,\ldots ,n-1\}$. 
	
	This amounts to a choice of $k$ integers. However, using the torus action the last row can be fixed to have all components equal to one, which brings down the number of choices to $k-1$. The fact that the matrix must have maximal rank requires all rows to be different and therefore a fixed point can be labeled by a $k-1$ tuple $\{ m_1,m_2,\ldots ,m_{k-1}\}$. Sometimes it would be convenient to use a $k$-tuple description where the $k^{\rm th}$ integer is set to be $m_k:= n $. We alternate between descriptions based on the application. 
	
	Finally, let us describe how a given $k-1$ tuple can generate $k-1$ equivalent ones. Consider a given choice $\{ \w_1,\w_2,\ldots ,\w_{k-1}\}$ of distinct $n^{\rm th}$ roots of unity corresponding to the choice $\{ m_1,m_2,\ldots ,m_{k-1}\}$.  The configuration of points on $\mathbb{CP}^{k-1}$ is then given by the $k\times n$ matrix with the $i^{th}$ column defined as
	\be
	\left( 1, \w_1^{(i-1)},\w_2^{(i-1)},\ldots ,\w_{k-1}^{(i-1)} \right)^T.
	\ee
	Let us choose any value $b\in \{1,2,\ldots ,k-1\}$ and use the torus action to rescale all columns as follows: Rescale the $i^{th}$ column by $(1/w_b)^{(i-1)}$. This has the effect of setting to $1$ the $b^{th}$ row of the $k\times n$ matrix defining the point on $X(k,n)$. Using a $SL(k)$ transformation to permute the rows we can send the row with all $1$'s to be the first one. This leads to a new matrix defining the {\it same} configuration of points in $\mathbb{CP}^{k-1}$ but with different values of integers. Moreover it is easy to find the new set of integers%\footnote{It is hard to overlook the fact that the transformation which leads to the new $k-1$ tuple has the exact form of a mutation in the quiver description of cluster algebras or of Seiberg duality in the physics context.}
	\be\label{clu}
	m_a \to \left\{ \begin{array}{cl}
		m_a - m_b & \;\;{\rm for} \; a\neq b, \\
		-m_a & \;\;{\rm for}\; a = b. 
	\end{array} \right.
	\ee
	When $n$ is prime this process groups all possibilities into
	\be\label{primeC}
	\frac{1}{k}\binomial{n-1}{k-1}
	\ee
	classes. 
	
	When $n$ is not prime the transformation \eqref{clu} does not necessarily produce distinct tuples and the number of classes has a structure that depends on the divisors of $n$. Clearly \eqref{primeC} provides an upper bound.  Indeed, when $n$ is not prime a tighter upper bound can be obtained.  
	
	To this end, in Section \ref{sec:enumeration cyclic fixed points}, to which we refer for details and expanded discussions, we give an injection into the set of certain aperiodic $k$-element subsets of $\{1,\ldots,n \}$ and we give a constructive upper bound for the number of cyclic fixed points in $\overline{X}(k,n)$.

	%%%%%%%%%%%%%%%%%%%%%%%%%%%%%%%%%%%%%
	
	\section{Critical Points of ${\cal S}_{k,n}$ on Planar Kinematics}\label{sec: critical points}
	
	%%%%%%%%%%%%%%%%%%%%%%%%%%%%%%%%%%%%%%
	
	In this section we study critical points of the function on $X(k,n)$ which is used in the definition of generalized biadjoint amplitudes when evaluated on planar kinematics,
	\be\label{eq: planar kinematics scattering equations}
	\spk = \sum_{i=1}^n \log \left( \frac{p_{i,i+1,\ldots i+(k-2),i+(k-1)}}{p_{i,i+1,\ldots i+(k-2),i+k}} \right).
	\ee
	Here all indices are defined modulo $n$ and $p_{i_1,i_2,\ldots ,i_k}$ denotes the minor of a $k\times n$ matrix representative of a point in $\overline{X}(k,n)$ made from columns $\{ i_1,i_2,\ldots ,i_k\}$. Note that the notation differs from the one the introduction \eqref{introPot} which was written in terms of Plucker coordinates of $G(k,n)$ in which the torus variables are exhibited explicitly but as it is well-known they completely drop out.  
	
	The aim of this section is to give a physically intuitive reason for why the critical points of $\spk$ are the cyclic fixed points discussed in the previous section. A formal proof requires making a connection to the superpotential of the mirror of the Grassmannian $G(n-k,n)$ and it is postponed to Section \ref{sec: superpotential PK potential equivalence}. 
	
	If $\spk$ is taken as a potential function describing the interaction of particles then particle at point $a$ only interacts with particles with indices in a range determined by the value of $k$. For example, if $k=2$, particle at point $a$ only interacts with particles at points $a-1$ and $a+1$. This is known as a nearest neighbor interaction if particles are thought of as spins on in a periodic one dimensional chain. The analogy with a spin chain is stronger if we allow each spin to carry degrees of freedom in $\mathbb{CP}^{k-1}$.
	
	In order to study the critical points of $\spk$ it is convenient to use a strategy familiar in statistical mechanics. We first consider the problem of an infinite number of spins on a line and then find solutions which satisfy the correct periodic boundary conditions to be interpreted as solutions to the problem of $n$ spins on a circle. 
	
	\subsection{Solving the Infinite Chain}
	
	Let us define the case of an infinite chain as that given by 
	\be
	{\cal S}_{k,\infty } := \sum_{i=-\infty}^\infty \log \left(\frac{p_{i,i+1,\ldots i+(k-2),i+(k-1)}}{p_{i,i+1,\ldots i+(k-2),i+k}} \right).
	\ee
	
	In this function, the indices are allowed to run over all integers. It is only when we restrict to finite $n$ that indices will be defined modulo $n$.
	
	Consider inhomogeneous variables in $\mathbb{CP}^{k-1}$ given by $(1,x_1,x_2,\ldots ,x_{k-1})$. When denoting a particular point we use $(1,x_1^{(i)},x_2^{(i)},\ldots ,x_{k-1}^{(i)})$.
	
	\begin{prop}\label{sol}
		Any $(k-1)$-tuple, $\{ \w_1,\w_2,\ldots ,\w_{k-1}\}$, of non-zero and distinct complex numbers defines a critical point of ${\cal S}_{k,\infty }$ given by $x_a^{(i)} = \w_a^{i-1}$.
	\end{prop}
	
	As mentioned above, the aim of this section is to give an intuitive reason for the Proposition. We do so by giving an elementary technique that we have used to prove it for values of $k < 8$.  
	
	Let us illustrate the idea with the $k=2$ case. We can simplify the notation and set $x_1^{(i)}=y_i$. The critical points are obtained by setting to zero the derivative of the potential function
	\be\label{k2case}
	\frac{\partial {\cal S}_{2,\infty }}{\partial y_i} = -\frac{1}{y_{i-1}-y_i} +\frac{1}{y_{i}-y_{i+1}} + \frac{1}{y_{i-2}-y_i}-\frac{1}{y_{i}-y_{i+2}}.
	\ee
	In order to verify Proposition \eqref{sol}, we set $y_i = \w^{i-1}$ and substitute it into \eqref{k2case} to get
	\be
	\frac{\partial {\cal S}_{2,\infty }}{\partial y_i} = \frac{1}{\w^{(i-2)}(1-\w)}\left( -\frac{1}{\w} +\frac{1}{\w^2 }+\frac{1}{1+\w}-\frac{1}{\w^2(1+\w)} \right). 
	\ee
	Combining the first two terms and the last two terms one finds
	\be
	\frac{\partial {\cal S}_{2,\infty }}{\partial y_i} = \frac{1}{\w^{(i-2)}(1-\w)}\left(\frac{1-\w}{\w^2}+\frac{\w-1}{\w^2} \right) = 0.
	\ee
	It is important to notice that the cancellation leading to the vanishing result is $i$ independent. 
	
	The reason we have shown the cancellation in detail here is that it hints that for general $k$ the key is to combine terms pairwise hoping that a telescopic cancellation would take place. It turns out that this is indeed the case.
	
	We have studied all cases up to $k=7$ and found the telescopic cancellation. Consider the $k=7$ case. In order to simplify the discussion, let us introduce another variable called $z$ to write the coordinate of the $i^{\rm th}$ point as $(z^{i-1},\w_1^{i-1},\w_2^{i-1},\ldots ,\w_{k-1}^{i-1})$. Let us study the scattering equation  
	\be
	\frac{\partial {\cal S}_{7,\infty }}{\partial z} = 0. 
	\ee
	The left hand since can be computed combining the contributions from each of the seven pairs. The result of the simplification of each pair is proportional to
	\be 
	\{ s_1 z^5,-s_1 z^5-s_2z^4,s_2z^4+s_3z^3,-s_3z^3-s_4z^2,s_4z^2+s_5z,-s_5z-s_6,s_6\}
	\ee
	where $s_m$ is the elementary symmetric polynomial of degree $m$ in $\{\w_1,\w_2,\ldots, \w_6\}$. The factor we have omitted depends on $i$ but since it is a common factor it is not relevant for the computation. Adding the terms shows that the cancellation is telescopic as expected.
	
	\subsection{Finite Chain}\label{sec:finite chain}
	
	We are interested making the infinite chain periodic with period $n$. Therefore we must impose that the $\mathbb{CP}^{k-1}$ value assigned to the $i^{\rm th}$ site is the same as that assigned to any $j\in \mathbb{Z}$ such that $j = i \mod n$. This has the effect of making the infinite chain equivalent to a chain on a circle with $n$ sites.
	
	It is clear that by requiring every entry of $\{ \w_1,\w_2,\ldots ,\w_{k-1} \}$ to be an $n$-root of unity the chain becomes periodic as desired. This leads to the following proposition.
	
	\begin{prop}\label{solFinite}
		Any $(k-1)$-tuple, $\{ \w_1,\w_2,\ldots ,\w_{k-1}\}$, of non-zero and distinct complex numbers, such that $\w_a^n= 1$ defines a critical point of $\spk$ given by $x_a^{(i)} = \w_a^{i-1}$ as long as none of the minors entering in $\spk$ vanishes.
	\end{prop}
	
	This result is still not satisfactory as the potential function $\spk$ is defined on the configuration space $\overline{X}(k,n)$ and distinct choices can lead to the same point in $\overline{X}(k,n)$. Moreover, these are clearly the fixed points of the cyclic action defined in section 2. This leads to the following refined statement, which is our main result.
	
	\begin{thm}\label{solFinite}
		Any fixed point of the cyclic action on $\overline{X}(k,n)$ defines a critical point of $\spk$ as long as none of the minors entering in $\spk$ vanishes.
	\end{thm}
	See the end of Section \ref{sec: superpotential PK potential equivalence} for the proof of Theorem \ref{solFinite}.

	In order to compute the number of solutions to the scattering equations it is useful to find a simple criterion to find out the fixed points that make at least one of the minors in $\spk$ vanish and the remove them. 
	
	Computing minors of the form $p_{i,i+1,\ldots ,i+k-2,i+k-1}$ one discovers that they cannot vanish if all roots of unity chosen are distinct. On the other hand $p_{i,i+1,\ldots ,i+k-2,i+k}$ vanishes non-trivially if and only if  
	\be\label{remove}
	1+\w_1+\w_2+\ldots +\w_{k-1} = 0.
	\ee
	This means that fixed points which satisfy \eqref{remove} are not solutions of the scattering equations and must be removed.
	
	The number of solutions, ${\cal N}_{k,n}$, is a very interesting function that depends on the factorization properties of $k$ and $n$. We have not been able to construct the function explicitly but in Section \ref{sec:enumeration cyclic fixed points} we prove an upper bound given by the number of binary Lyndon words.

	\section{Relation to Mirror Symmetry Superpotential}\label{sec: superpotential PK potential equivalence}
	
	In this section we prove the result stated in Theorem \ref{solFinite} by making a connection between the equations that determine the critical points of the CEGM potential, $\spk$, and those of the superpotential, ${\cal F}_q$, in the theory which is the mirror of the Grassmannian $G(n-k,n)$. The precise form we use is that introduced by Marsh and Rietsch in \cite{marsh2020b} as a rational function of the Plucker coordinates of $G(k,n)$.
	
	In \cite{karp2019moment} Karp proved that fixed points of the Grassmannian $G(k,n)$ under a certain cyclic shift map are the critical points of a superpotential ${\cal F}_q$. 
	
	We prove that by treating $G(k,n)$ as a torus fibration over $X(k,n)$, which is possible at least locally, and ``integrating out" the fields that control the scale of each point one finds that the critical points of ${\cal F}_q$ contain those of the  potential on planar kinematics, $\spk$.   
	
	Let $\Delta_{a_1,a_2,\ldots ,a_k}$ denote the minors of the $k\times n$ matrix representative of a point in $G(k,n)$, i.e. the Plucker coordinates.
	
	The mirror symmetry superpotential introduced by  Marsh and Rietsch \cite{marsh2020b} is
	\be
	{\cal F}_q := \sum_{i=1,\, i\neq n-k}^{n}\frac{\Delta_{i,i+1,\ldots, i+(k-2),i+k}}{\Delta_{i,i+1,\ldots, i+(k-2),i+(k-1)}} + q \frac{\Delta_{n-k,n-k+1,\ldots,n-1,1}}{\Delta_{n-k,n-k+1,\ldots, n-1,n}}.
	\ee
	This superpotential depends on a parameter $q$. In fact, one could include a parameter $q_i$ for each term but using a simple rescaling of the fields all parameters can be removed except for one, which by convention is chosen to be $q_{n-k}:=q$. In order to simply the notation in the computations below it is actually useful to keep the other parameters and write
	\be
	{\cal F} := \sum_{i=1}^{n}q_i\frac{\Delta_{i,i+1,\ldots, i+(k-2),i+k}}{\Delta_{i,i+1,\ldots, i+(k-2),i+(k-1)}}.
	\ee

	In order to proceed let us choose a chart of $G(k,n)$ parameterized as 
	\be\label{chartG}
	\left(
	\begin{array}{cccccccc}
		t_1 & 0  & \cdots & 0 & t_{k+1} & t_{k+2} & \ldots & t_{n}  \\
		0      &    t_2    & \cdots & 0 & t_{k+1} & t_{k+2}x_{1,1} & \ldots & t_{n}x_{1,n-k-1}  \\
		\vdots & \vdots  & \cdots & \vdots & \vdots &\vdots & & \vdots   \\
		0 & 0  &  \ddots & t_k & t_{k+1} & t_{k+2}x_{k-1,1}  & \ldots & t_{n}x_{k-1,n-k-1}
	\end{array}
	\right) . 
	\ee
	As usual, other charts might be necessary to cover all points of interest but the argument can be carried out in the exactly the same way. 
	
	The Plucker coordinates of $G(k,n)$ can now be written as
	\be
	\Delta_{a_1,a_2,\ldots ,a_k} = t_{a_1}t_{a_2}\ldots t_{a_k}p_{a_1,a_2,\ldots ,a_k}
	\ee
	where $p_{a_1,a_2,\ldots ,a_k}$ denote the minors of a matrix representative of a point in $\overline{X}(k,n)$, i.e. the minors of 
	\be\label{coord}
	\left(
	\begin{array}{cccccccc}
		1 & 0  & \cdots & 0 & 1 & 1 & \ldots & 1  \\
		0      &    1    & \cdots & 0 & 1 & x_{1,1} & \ldots & x_{1,n-k-1}  \\
		\vdots & \vdots  & \cdots & \vdots & \vdots &\vdots & & \vdots   \\
		0 & 0  &  \ddots & 1 & 1 & x_{k-1,1}  & \ldots & x_{k-1,n-k-1}
	\end{array}
	\right) . 
	\ee
	
	In this chart the superpotential becomes 
	\be
	{\cal F} = \sum_{i=1}^{n}q_i\frac{t_{i+k}\, p_{i,i+1,\ldots, i+(k-2),i+k}}{t_{i+(k-1)}\, p_{i,i+1,\ldots, i+(k-2),i+(k-1)}}.
	\ee
	Differentiating with respect to $t_a$ gives
	\be
	\frac{\partial {\cal F}}{\partial t_a} = -q_{a-k}\frac{1}{t_{a-1}}\frac{p_{\ldots,a-3, a-2,a}}{p_{\ldots ,a-3,a-2,a-1} }+q_{a-k+1}\frac{t_{a+1}}{t^2_{a}}\frac{p_{\ldots,a-2, a-1,a+1}}{p_{\ldots ,a-2,a-1,a}}.
	\ee
	Setting this to zero implies that 
	\be
	q_{a-k}\frac{t_a}{t_{a-1}}\frac{p_{\ldots,a-3, a-2,a}}{p_{\ldots ,a-3,a-2,a-1} } = q_{a-k+1}\frac{t_{a+1}}{t_{a}}\frac{p_{\ldots,a-2, a-1,a+1}}{p_{\ldots ,a-2,a-1,a}}.
	\ee
	and therefore the following quantity is independent of $a$,
	\be\label{eq:Gamma}
	\Gamma := q_{a-k}\frac{t_a}{t_{a-1}}\frac{p_{\ldots,a-3, a-2,a}}{p_{\ldots ,a-3,a-2,a-1}}.
	\ee
	Let us now compute the derivative of ${\cal F}$ with respect to any variable that appears in the minors $p_{b_1,b_2,\ldots ,b_k}$. Let us denote such a generic variables as $z_a$, then
	\be\label{hepo}
	\frac{\partial {\cal F}}{\partial z_a} =  q_a\frac{t_{a+k}}{t_{a+(k-1)}}\frac{\partial }{\partial z_a}\left(\frac{p_{a,a+1,\ldots a+(k-2),a+k}}{p_{a,a+1,\ldots ,a+(k-2),a+(k-1)}}\right) + \ldots 
	\ee
	Critical points are found by setting this to zero and provided $\Gamma$ does not vanish the equations are equivalent to 
	\be\label{deriv}
	\frac{1
	}{\Gamma}\frac{\partial {\cal F}}{\partial z_a} = 0 \qquad \forall\, a.
	\ee
	
	Using a form of $\Gamma$ appropriate to each term one can turn each term in the sum into a logarithmic derivative. For example, consider the contribution of the first term in \eqref{hepo} to the lhs of the equation in \eqref{deriv},
	\be
	\frac{q_a
	}{\Gamma}\frac{t_{a+k}}{t_{a+(k-1)}}\frac{\partial }{\partial z_a}\left(\frac{p_{a,a+1,\ldots a+(k-2),a+k}}{p_{a,a+1,\ldots ,a+(k-2),a+(k-1)}}\right).
	\ee
	Using 
	\be
	\Gamma = q_a\frac{t_{a+k}}{t_{a+(k-1)}} \frac{p_{a,a+1,\ldots a+(k-2),a+k}}{p_{a,a+1,\ldots ,a+(k-2),a+(k-1)}}
	\ee
	the expression simplifies to
	\be
	\frac{p_{a,a+1,\ldots ,a+(k-2),a+(k-1)}}{p_{a,a+1,\ldots a+(k-2),a+k}} \frac{\partial }{\partial z_a}\left(\frac{p_{a,a+1,\ldots a+(k-2),a+k}}{p_{a,a+1,\ldots ,a+(k-2),a+(k-1)}}\right) 
	\ee
	which can be written as
	\be
	\frac{\partial }{\partial z_a} \log\left(\frac{p_{a,a+1,\ldots a+(k-2),a+k}}{p_{a,a+1,\ldots ,a+(k-2),a+(k-1)}}\right).
	\ee
	Putting all together the equations for the superpotential \eqref{deriv} can be written as
	\be
	\frac{\partial }{\partial z_a} \sum_{i=1}^n\log\left(\frac{p_{i,i+1,\ldots i+(k-2),i+k}}{p_{i,i+1,\ldots ,i+(k-2),i+(k-1)}}\right) = 0 \qquad \forall\, a
	\ee
	which coincide with the scattering equations on planar kinematics.
	
	Note that the derivation is only valid if $\Gamma$ is not zero and so we have proven the following.
	\begin{lem}
		Any critical point of ${\cal F}_q$ for which $\Gamma$ does not vanish descends to a critical point of $\spk$ evaluated on planar kinematics. 
	\end{lem}

	Let us now discuss the critical points of ${\cal F}_q$.
	\begin{prop}[\cite{karp2019moment}]
		For $t\in \mathbb{C}^*$, the critical points of ${\cal F}_q$ on $G(k,n)$ at $q = t$ are precisely the fixed points of the t-deformed cyclic shift map $\sigma_t$.
	\end{prop}

	Here the t-deformed cyclic shift map, $\sigma_t$, is defined as a map $\mathbb{C}^n \to \mathbb{C}^n$ which acts on $G(k,n)$ by acting on each of the rows of a $k\times n$ matrix representative. The precise definition of $\sigma_t$ is the following.
	\begin{defn}[\cite{karp2019moment}]\label{karpSigma}
		For $t\in \mathbb{C}^*$, define the t-deformed (left) cyclic shift map $\sigma_t\in GL(n,\mathbb{C})$ by
		$$ \sigma_t(v) = (v_2,v_3,\cdots ,v_n,(-1)^{k-1}t v_1)\quad {\rm for} \quad v=(v_1,v_2,\cdots ,v_n) \in \mathbb{C}^{n}. $$
	\end{defn}
	%
	% \subsection{Cyclic Fixed}
	Finally, we are ready to prove our Theorem \ref{solFinite}.  For convenience we state it again.
	\begin{thm}%\label{solFinite}
		Any fixed point of the cyclic action on $\overline{X}(k,n)$ defines a critical point of $\spk$ as long as none of the minors entering in $\spk$ vanishes.
	\end{thm}
	
	\begin{proof}
		First note that fixed points of the t-deformed cyclic shift map $\sigma_t$ clearly descend to fixed points of our cyclic action $\rho_n$. Of course, two fixed points in $G(k,n)$ which only differ by a torus action descend to the same fixed point in $X(k,n)$. However, such fixed points only produce solution to the scattering equations if $\Gamma\neq 0$. Since the scale factors $t_a$ do not vanish in any of the cyclic fixed points in $G(k,n)$, the only way $\Gamma$ can vanish is for cyclic fixed points for which $p_{a,a+1,\ldots ,a+k-2,a+k} = 0$. But these are exactly the fixed points excluded in the Proposition.
		
	\end{proof}
	
	\section{Enumeration of Aperiodic Critical Points of the Planar Kinematics Potential Function}\label{sec:enumeration cyclic fixed points}
	In this section, our aim is to give an explicit combinatorial tabulation of the critical points of  $\mathcal{S}^{(PK)}_{k,n}$; we do not completely succeed, but we are able to give a useful constructive upper bound. 
	
	In what follows, it is convenient to regard $\mathcal{S}^{(PK)}_{k,n}$ as a function on the complex Grassmannian that happens to be invariant under not only the action of the torus group $(\mathbb{C}^\ast)^n$ (denoted $\mathbf{T}$ in Section \ref{sec: cyclic shift}), but in fact it is invariant under the action of the semidirect product $ (\mathbb{C}^\ast)^n \rtimes \mathbb{Z}\slash n \hookrightarrow GL(n)$, where the subgroup $(\mathbb{C}^\ast)^n$ acts by scaling the standard basis vectors in $\mathbb{C}^n$ by complex numbers in the standard way as $\lambda \cdot e_j = \lambda_j e_j$, and the subgroup $\mathbb{Z}\slash n$ acts by the cyclic rotation operator $\rho_n(e_j) = e_{j-1}$.  
	Given $2\le k\le n-2$, put $q=\exp(2\pi i/n)$.
	
	Let $T_{n} \simeq\mathbb{Z}\slash n$ be the subgroup of $GL(n)$ embedded into the diagonal as
	$$a\mapsto \text{diag}(1,q^a q^{2a},\ldots, q^{(n-1)a}).$$
	
	Then in particular, $T_n$ acts on the standard basis of $\mathbb{C}^n$ by 
	$$a:e_j \mapsto q^{a(j-1)}e_j.$$

	Denote by $\binom{\lbrack n\rbrack}{k}$ the set of $k$-element subsets of $\lbrack n\rbrack = \{1,\ldots ,n\}$; then the group $\mathbb{Z}\slash n$ acts by $\{j_1,\ldots, j_k\} \mapsto \{j_1+a,\ldots, j_k+a\}$.
	
	\begin{defn}
		We say that a k-element subset $J = \{j_1,\ldots, j_k\} \in \binom{\lbrack n\rbrack}{k}$ is \textit{aperiodic} if its $\mathbb{Z}\slash n$-orbit has exactly $n$ elements,
		$$\left\vert\left\{\{j_1+j,\ldots, j_k+j\}: j\in\mathbb{Z}\slash n \right\} \right\vert = n,$$
		where addition is regarded modulo $n$.
	\end{defn}
	
	Recall that a string $w$ with $k$ ones and $n-k$ zeros is a binary \textit{Lyndon word} if it is the unique lexicographically smallest element among its cyclic rotations.  As it is the unique lexicographically smallest element among its cyclic rotations, it follows that $w$ is different from its cyclic rotations.
	
	Recall that the number of binary Lyndon words with $k$ ones and $n-k$ zeros is equal to 
	\begin{eqnarray}\label{eq:Lyndon word enumeration}
		\mathcal{N}_{k,n} & = & \frac{1}{n}\sum_{d | {\rm gcd}(k,n)}\left( \mu(d)\binom{n/d}{k/d} \right),
	\end{eqnarray}
	see O.E.I.S. number triangle A051168 \cite{oeis}.
	
	Here $\mu(d)$ is the Moebius function,
	$$\mu(d) = 	\begin{cases}
		0, & \text{if $d$ is a product of primes with repeated factors,}\\
		1, & d=1,\\
		(-1)^\ell & \text{if $d$ is a product of $\ell$ distinct primes.}
	\end{cases}$$
	\begin{prop}
		The number of equivalence classes of aperiodic $k$-element subsets of $\lbrack n\rbrack$ modulo $\mathbb{Z}\slash n$, is given by Equation \eqref{eq:Lyndon word enumeration}.
	\end{prop}
	
	\begin{proof}
		This is a straightforward consequence of the standard bijection between $\mathbb{Z}\slash n$-orbits of $k$-element subsets of $\lbrack n\rbrack$ and Lyndon words.  For the bijection, one identifies a subset $J = \{j_1,\ldots, j_k\} \in \binom{\lbrack n\rbrack}{k}$ with its indicator function $e_J = \sum_{j\in J}e_j$; then restrict to lexicographically minimal indicator functions.
		
		Here $\mathbb{Z}\slash n$ acts on binary Lyndon words via the $n$-cycle $(12\cdots n)$ on positions, while it acts on aperiodic subsets by permuting index labels.
	\end{proof}
	
	Given $A=\left\{a_1,\ldots, a_k \right\} \in \binom{\lbrack n\rbrack}{k}$, then clearly since the elements $a_1,\ldots, a_k$ are distinct, the $k\times n$ matrix $g_A$, which we define by its entries $x_{i,j} = q^{(j-1)a_i}$, has nonvanishing minor $p_{1,2,\ldots, k-1,k}(g_A)$, (which is the Vandermonde determinant in the entries $q^{a_1},\ldots, q^{a_k}$) so it has rank $k$ and defines an element of $G(k,n)$.
	
	Let us call a cyclic fixed point $g\in G(k,n)$ \textit{aperiodic} if its $T_n$-orbit has exactly $n$ distinct cyclic fixed points.  It follows immediately that the number of $T_n$-orbits of aperiodic cyclic fixed points in $G(k,n)$ is given by $\mathcal{N}_{k,n}$ in Equation \eqref{eq:Lyndon word enumeration}.
	
	In Theorem \ref{thm: Injection}, we show that the minors $p_{i,i+1,\ldots, i+k-2,i+k}$ that appear in the planar kinematics potential function $\mathcal{S}^{(PK)}_{k,n}$ vanish on $T_n$-orbits which are \textit{not} aperiodic; however, there will be defective cyclic fixed points $g\in G(k,n)$ which are aperiodic, but for which we still have $p_{i,i+1,\ldots, i+k-2,i+k}(g) = 0$.  Recall that these minors appeared in the factor $\Gamma$ in Equation \eqref{eq:Gamma}, which was assumed to be nonzero.

	\begin{thm}\label{thm: Injection}
		The set of cyclic fixed points in $X(k,n)$ injects into the set of $T_n$-orbits of aperiodic cyclic fixed points in $G(k,n)$.
	\end{thm}
	
	\begin{proof}
		Supposing that $g\in G(k,n)$ is any cyclic fixed point.  Then it follows from \cite[Theorem 1.1]{karp2019moment} that there exists a unique $A = \{a_1,\ldots, a_k\} \in \binom{\lbrack n\rbrack}{k}$ such that $g = g_A$ modulo $GL(k)$.
		
		Let us suppose that $A$ were not aperiodic; this means that there exists $m \in \{1,\ldots, n-1\}$ such that as sets we have
		$$\{a_1+m,\ldots, a_k+m\} =\{a_1,\ldots, a_k\}.$$
		First note that
		$$	\frac{p_{1,2,\ldots, k-1,k+1}(g_A)}{p_{1,2,\ldots, k-1,k}(g_A)} = q^{a_1}+q^{a_2}+\cdots +q^{a_k}.$$
		Then we have
		\begin{eqnarray*}
			q^m\left(q^{a_1}+q^{a_2} + \cdots  + q^{a_k}\right)	& = & q^{a_1} + q^{a_2} + \cdots +q^{a_k}
		\end{eqnarray*}
		hence
		\begin{eqnarray*}
			0 & = & (1-q^m)\left( q^{a_{1+t}} + q^{a_{2+t}} + \cdots + q^{a_{k+t}}\right)\\
			\Rightarrow  0 & = & q^{a_1} + q^{a_2} + \cdots + q^{a_k}.
		\end{eqnarray*}
		This implies that $g$ cannot be a critical point of the planar kinematics potential function.
		
	\end{proof}
	
	Consequently we finally obtain our constructive upper bound on the number of critical points of the planar kinematics potential function \eqref{eq: planar kinematics scattering equations}.
	\begin{cor}
		For any $2\le k\le n-2$, then the PK potential function $\mathcal{S}^{(PK)}_{k,n}$ has at most $\mathcal{N}_{k,n}$ critical points, where $\mathcal{N}_{k,n}$ is the number of Lyndon words with $k$ ones and $n-k$ zeros, given in Equation \eqref{eq:Lyndon word enumeration}.
	\end{cor}
	
	Below we enumerate cyclic equivalence classes of aperiodic $k$-element subsets of $\{1,\ldots, n\}$, that is to say, Lyndon words with $k$ one's and $n-k$ zero's, for $k\le 6$ and $n\le 24$.  The numbers of critical points are given subsequently.
	$$
	\begin{array}{c|cccccccccccccccccccccc}
		k\backslash n & 3 & 4 & 5 & 6 & 7 & 8 & 9 & 10 & 11 & 12 & 13 & 14 & 15 & 16 & 17 & 18 & 19 & 20 & 21 & 22 & 23 & 24 \\
		\hline
		2 & 1 & 1 & 2 & 2 & 3 & 3 & 4 & 4 & 5 & 5 & 6 & 6 & 7 & 7 & 8 & 8 & 9 & 9 & 10 & 10 & 11 & 11 \\
		3 & 0 & 1 & 2 & 3 & 5 & 7 & 9 & 12 & 15 & 18 & 22 & 26 & 30 & 35 & 40 & 45 & 51 & 57 & 63 & 70 & 77 & 84 \\
		4 & 0 & 0 & 1 & 2 & 5 & 8 & 14 & 20 & 30 & 40 & 55 & 70 & 91 & 112 & 140 & 168 & 204 & 240 & 285 & 330 & 385 & 440 \\
		5 & 0 & 0 & 0 & 1 & 3 & 7 & 14 & 25 & 42 & 66 & 99 & 143 & 200 & 273 & 364 & 476 & 612 & 775 & 969 & 1197 & 1463 & 1771 \\
		6 & 0 & 0 & 0 & 0 & 1 & 3 & 9 & 20 & 42 & 75 & 132 & 212 & 333 & 497 & 728 & 1026 & 1428 & 1932 & 2583 & 3384 & 4389 & 5598 \\
	\end{array}
	$$
	Call a cyclic fixed point $g_A\in X(k,n)$ \textit{defective} if $A\in \binom{\lbrack n\rbrack}{k}$ is aperiodic, but we still have
	$$\det(v_i,\ldots, v_{i+k-2},v_{i+k})=0$$
	for $i=1,\ldots, n$.
	
	In other words, $g_A$ does not define a solution to the scattering equations at the PK point.
	
	We see this behavior for the first time at k=5.
	
	In the table above the actual number of critical points is less than the number of Lyndon words starting at $k=5$, where the (nonzero) entries are now given by 
	$$
	\begin{array}{ccccccccccccccccccc}
		1 & 3 & 7 & 14 & 25 & 42 & 65 & 99 & 143 & 200 & 273 & 364 & 474 & 612 & 775 & 969 & 1197 & 1463 & 1768. \\
	\end{array}
	$$

	In what follows, we tabulate representatives of the first few defective aperiodic cyclic fixed points which are not critical points.  %Note that in the notation below we are keeping the last entry, $n$.  
	
	k=5: 
	\begin{eqnarray*}\label{eqn:aberrant necklaces}
		n=12:& &  \{(1,4,7,8,12)\}\}\\
		n=18: & & \{(1,6,10,12,18),(1,7,9,13,18)\}\\
		n=24: & & \{(1,8,13,16,24),(1,9,12,17,24),(2,8,14,16,24)\}\\
		n=30: & & \{(1,10,16,20,30),(1,11,15,21,30),(2,10,17,20,30),(2,12,15,22,30)\}.
	\end{eqnarray*}
	Thus, the count decreases by $\frac{n-6}{6}$ for $n=12,18,24,30,36,\ldots$.  We have checked that this formula holds through $n=90$.
	
	For instance, for $n=12$ we have
	$$q+q^4+q^7+q^8+q^{12}=0,$$
	where $q=\exp(2\pi i/12)$.

	Also, by explicit computation, for $k=6$ one finds exactly one defective aperiodic cyclic fixed point at $n=30$ and one at $n=60$; we did not attempt to compute larger $n$.  These correspond to
	\begin{eqnarray*}\label{eqn:aberrant necklaces 2}
		n=30: & & \{(1,7,13,19,20,30)\}\\
		n=60: & & \{(2,14,26,38,40,60)\}.
	\end{eqnarray*}
	Based on the data for $n=5,6$, it is tempting to try to refine the upper bound to an exact enumeration; but finding the general rule for all $2\le k \le n-2$ appears to be beyond the scope of this paper and is left to future work.

	\section{Evaluating CEGM Biadjoint Amplitudes}\label{sec: Evaluating CEGM Biadjoint Amplitudes}
	
	In this section we review the construction of CEGM biadjoint amplitudes with special detail on the $SL(k)$ gauge fixing procedure. In fact, on planar kinematics there are solutions which do not admit the standard gauge fixing and therefore more general gauge fixings are necessary. 
	
	Recall that the most general $\mathbb{CP}^{k-1}$ scattering equations are the conditions for finding the critical points of a general potential function 
	\be
	{\cal S}_{k,n} = \sum_{b_1,b_2,\ldots ,b_k=1}^n \s_{b_1,b_2,\ldots ,b_k} \log p_{b_1,b_2,\ldots ,b_k}.
	\ee
	More explicitly,
	
	\be\label{scattering-equations}
	\frac{\partial {\cal S}_{k,n}}{\partial \x_{a,i}} = 0 \qquad \forall\,\, (a,i),
	\ee
	where $\x_{a,i}$ represent inhomogeneous coordinates of the $a^{\rm th}$ point on $\mathbb{CP}^{k-1}$. The coordinates can be arranged in a matrix
	\be\label{coord}
	\left(
	\begin{array}{cccccc}
		1 & 1       & \cdots & 1 & 1  \\
		\x_{1,1} & \x_{2,1}     & \cdots & \x_{n-1,1} & \x_{n,1}  \\
		\x_{1,2} & \x_{2,2} & \cdots & \x_{n-1,2} & \x_{n,2}  \\
		\vdots & \vdots &  \ddots       & \vdots & \vdots \\
		\x_{1,k-1} & \x_{2,k-1} & \cdots & \x_{n-1,k-1} & \x_{n,k-1} \\
	\end{array}
	\right) . 
	\ee
	
	In order for the potential function to be well-defined on $X(k,n)$ the kinematic invariants $\s_I$ must be completely symmetric in their indices and satisfy the following properties:
	\be
	\sum_{b_2,b_3,\ldots ,b_n =1}^n\s_{a,b_2,b_3,\ldots ,b_n} =0 \quad {\rm and} \quad  \s_{a,a,b_3,\ldots ,b_k} = 0 \quad \forall \, a\in \{1,2,\ldots ,n\}.
	\ee
	
	The set of scattering equations \eqref{scattering-equations} is covariant under the action of $SL(k)$ acting on the matrix \eqref{coord} by left multiplication. This means that $k^2-1$ equations are redundant. This is a welcome fact as $SL(k)$ can be used to fix $k^2-1$ of the variables in the matrix \eqref{coord}. These two facts mean that the Hessian matrix of ${\cal S}_{k,n}$ which is a $(k-1)n \times (k-1)n$ matrix has corank $k^2-1$. 
	
	The evaluation of the amplitudes requires the definition of a reduced determinant of the Hessian matrix, since the Hessian of ${\cal S}_{k,n}$ is the Jacobian matrix of the scattering equations.
	
	In the CEGM original work, the reduced determinant was defined by analogy with the well-known $k=2$ case. Let us describe such particular construction before discussing the most general one. 
	
	The components of the Hessian in this context are usually denoted $\Psi_{IJ}$, with composed indices $I=(a,i)$ and $J=(b,j)$ so that 
	\be
	\Psi_{IJ}:=\frac{\partial^2 {\cal S}_{k,n}}{\partial \x_{a,i}\partial \x_{b,j}}. 
	\ee
	The CEGM construction of the reduced determinant is defined by selecting a submatrix obtained from $\Psi$ by deleting $k^2-1$ rows and $k^2-1$ columns, computing its determinant and compensating with a factor which makes the object independent of the choices made. Let us denote the submatrix obtained by deleting all rows that contain labels $\{ a_1,a_2\ldots a_{k+1}\}$ in their indices; a total of $(k-1)(k+1)$, and rows containing labels $\{ b_1,b_2\ldots b_{k+1}\}$ in their indices by $\Psi^{a_1,a_2\ldots,a_{k+1}}_{b_1,b_2,\ldots, b_{k+1}}$. Then the reduced determinant is 
	\be\label{CEGM def}
	{\rm det}'\Psi^{(k)} :=\frac{{\rm det}\Psi^{a_1,a_2\ldots,a_{k+1}}_{b_1,b_2,\ldots, b_{k+1}}}{\V_{a_1,a_2,\ldots ,a_{k+1}}\V_{b_1,b_2,\ldots, b_{k+1}}},
	\ee
	where the $\V_{a_1,a_2,\ldots ,a_{k+1}}$ is a generalization of a Vandermonde determinant defined by 
	\be\label{general Vander}
	\V_{a_1,a_2,\ldots ,a_{k+1}}:= \prod_{i=1}^{k+1}p_{a_1,a_2,\ldots ,\hat{a}_{i},\ldots a_{k+1}}.
	\ee
	
	Clearly, this definition of the reduced determinant requires 
	$\V_{a_1,a_2,\ldots ,a_{k+1}}$ and $\V_{b_1,b_2,\ldots, b_{k+1}}$ to be non-vanishing on the solution to the scattering equations used in the evaluation. Since the choice of the sets $\{a_1,a_2,\ldots ,a_{k+1}\}$ and $\{b_1,b_2,\ldots, b_{k+1}\}$ is arbitrary, one can try different choices until the generalized Vandermonde determinants are non-vanishing.
	
	\begin{defn}\label{frame}
		A set $\{a_1,a_2,\ldots ,a_{k+1}\}$ is called an $SL(k)$ frame on a particular solution to the scattering equations if the corresponding generalized Vandermonde determinant, $\V_{a_1,a_2,\ldots ,a_{k+1}}$, evaluated on the solution is non-zero.
	\end{defn}
	
	Now we can restate the applicability of the CEGM definition of reduced determinant. Formula \eqref{CEGM def} can be used on a given solution to the scattering equations if and only if the solution defines a point in $\overline{X}_{k,n}$ with at least one frame. 
	
	When $k=2$ all solutions to the scattering equations admit at least one frame. However, in the next section we find that $k=4,n=9$ is the first case with {\it frameless} solutions.
	
	When dealing with frameless solutions one has to use a more general gauge fixing procedure. Since $k=4$ is our main application in this work, we describe the construction in that case and leave the general $k$ construction as a straightforward exercise to the reader.
	
	\subsection{General $SL(4)$ Gauge Fixing}\label{general gauge}
	
	Consider an arbitrary infinitesimal $SL(4)$ transformation acting on a point in $\mathbb{CP}^3$. Let us parameterize the transformations as
	\be\label{four}
	\left(
	\begin{array}{cccc}
		1+\e_{11} & \e_{12}  & \e_{13} & \e_{14}  \\
		\e_{21} & 1-\e_{11}+\e_{22}  & \e_{23} & \e_{24}  \\
		\e_{31} & \e_{32}  & 1-\e_{22}+\e_{33} & \e_{34}  \\
		\e_{41} & \e_{42}  & \e_{43} & 1 - \e_{33} \\
	\end{array}
	\right) . 
	\ee
	Here $\epsilon_{ij}$ are infinitesimal deformations and we have chosen to impose the tracelessness condition of the infinitesimal generations in a particular way. There are $4^2-1=15$ infinitesimal deformations.
	
	To obtain the action on $(1,x_1,x_2,x_{3})^T$ we simply multiply on the left by \eqref{four} and use the torus action to set the top component to one. Performing this and subtracting the original vector one finds the infinitesimal variations
	\begin{align}\label{gen4}
		& & \nonumber \delta x_1  =  -x_1^2 \e_{12}-2 x_1 \e_{11}+x_1 \e_{22}-x_2 x_1
		\e_{13}-x_3 x_1 \e_{14}+x_2 \e_{23}+x_{3} \e_{24}+\e_{21}, \\  
		\nonumber & &  \delta x_2  =  x_1 \e_{32}-x_2 \e_{22}+x_2 \e_{33}+x_3 \e_{34}+x_2 (-x_1
		\e_{12}-x_2 \e_{13}-x_3 \e_{14}-\e_{11})+\e_{31}, \\
		& & \delta x_3  = x_1 \e_{42}+x_2\e_{43}-x_3 \e_{33}+x_3 (-x_1 \e_{12}-x_2
		\e_{13}-x_3 \e_{14}-\e_{11})+\e_{41}.  
	\end{align}
	The key idea is that these infinitesimal variations provide a way of computing a basis of the null space of the Jacobian matrix which is covariant under $SL(k)$ and torus actions. The null space is spanned by $15$ vectors in $\mathbb{C}^{3n}$. There is one vector for each $\e_{i,j}$. For example, consider $\e_{4,1}$. Setting all other $\e_{i,j}$ to zero in \eqref{gen4}
	\be
	(\delta x_1,\delta x_2,\delta x_3) = \left(-x_1 x_2,-x_2 x_3,-x_3^2\right)\e_{4,1}.
	\ee
	Applying this to the coordinates of all $n$ particles produces a $3n$ dimensional vector
	\be
	v_{41}:=\left(-x_{1,1} x_{2,1},-x_{2,1} x_{3,1},-x_{3,1}^2,\ldots ,-x_{1,n} x_{2,n},-x_{2,n} x_{3,n},-x_{3,n}^2 \right)^T.
	\ee

	These vectors can be grouped into a $15\times 3n$ matrix, 
	\be
	{\cal V} := \left( v_{11},v_{22},v_{33},v_{12},\ldots ,v_{43}\right).
	\ee
	Note that we have not yet fixed the normalization of the vectors spanning the null space. This is done when we reproduce the standard CEGM gauge fixing.

	It is convenient to give a notation for the minors of $\cal V$. Recall that the entries of the the Hessian matrix, $\Psi_{IJ}$, where indexed with $I=(a,i)$ and $(J,b)$. Here we allow $I$ and $J$ to be numbers from $1$ to $3n$ with the matching made lexicographically to $(a,i)$. For example, $I=4$ corresponds to $(2,1)$. The minor of $\cal V$ made with rows $\{I_1,I_2,\ldots I_{15}\}$ is denoted as $[I_1,I_2,\ldots ,I_{15}]$.
	
	Now we are ready to define the most general $SL(4)$ gauge fixing and its associated reduced determinant,
	\be\label{general detPrime}
	{\rm det}'\Psi^{(4)} :=\frac{ {\cal N}\,{\rm det}\Psi^{I_1,I_2,\ldots ,I_{15}}_{J_1,J_2,\ldots ,J_{15}}}{[I_1,I_2,\ldots ,I_{15}][J_1,J_2,\ldots ,J_{15}]}.
	\ee
	Here ${\cal N}$ is a proportionality constant which is needed to match the normalization of biadjoint amplitudes. If desired, ${\cal N}$ could be reabsorbed in the normalization of the vectors chosen to span the null space of the Hessian.   
	
	%as it becomes clear in Proposition \eqref{vanProp}.
	
	\begin{prop}
		The value of ${\rm det}'\Psi^{(4)}$ evaluated on a solution to the scattering equations is independent of the choice of sets $\{I_1,I_2,\ldots I_{15}\}$ and $\{J_1,J_2,\ldots J_{15}\}$, up to a sign, for all choices in which neither  
		$[I_1,I_2,\ldots ,I_{15}]$ nor $[J_1,J_2,\ldots ,J_{15}]$ vanish.
	\end{prop}
	
	The proof is a simple extension of the one given in Appendix A of \cite{Cachazo:2012pz} for the $k=2$ case.
	
	Let us end this part of the section with a discussion on how to recover the CEGM gauge fixing from the generalized one and in the process we fix the normalization ${\cal N}$. In cases in which there is a frame, it is natural to select
	$\{I_1,I_2,\ldots I_{15}\}$ so that they agree with
	\be\label{partC}
	\{ (a_1,1),(a_1,2),(a_1,3),(a_2,1),(a_2,2),(a_2,3)\ldots , (a_5,1),(a_5,2),(a_5,3)\}.
	\ee
	In other words, one selects five particle labels $\{a_1,a_2,a_3,a_4,a_5\}$ and all three coordinates for each.
	
	\begin{prop}\label{vanProp}
		Given a choice of $\{I_1,I_2,\ldots ,I_{15}\}$ as in \eqref{partC} the following agree
		\be\label{idenV}
		[I_1,I_2,\ldots ,I_{15}] = 4\V_{a_1,a_2,a_3,a_4,a_5}
		\ee
		with $\V$ defined in \eqref{general Vander} as
		$$ \V_{a_1,a_2,a_3,a_4,a_5} = \prod_{i=1}^{5}p_{a_1,a_2,\ldots ,\hat{a}_{i},\ldots a_{5}}. $$
	\end{prop}
	
	The proof is easily carried out using a symbolic manipulation program. The identity \eqref{idenV} is purely algebraic and it does not require to be on the support of the scattering equations.
	
	Using this result in the definition of the reduced determinant \eqref{general detPrime} one immediately concludes that ${\cal N}=4^2=16$.

	\section{CEGM Amplitudes on Planar Kinematics}\label{sec:CEGM Amplitudes on Planar Kinematics}
	
	In this section we evaluate the CEGM biadjoint amplitudes on planar kinematics in order to provide support for the conjecture stating that their values are computed by the higher dimensional Catalan numbers. 
	
	In order to evaluate the CEGM biadjoint amplitudes on the planar kinematics it is necessary to introduce the $k$-Parke-Taylor factor, 
	\be
	{\rm PT}(1,2,\ldots ,n) := \frac{1}{p_{1,2,\ldots ,k}\, p_{2,3,\ldots ,k+1}\cdots p_{n,1,\ldots ,k-1}}.
	\ee
	Finally, the CHY formulation of the CEGM biadjoint amplitude is constructed as follows
	\be
	m_n^{(k)}(\mathbb{I},\mathbb{I})=\sum_{m=1}^{{\cal N}_{n,k}}\left.\frac{1}{{\rm det}'\Psi^{(k)}}\, \left({\rm PT}(1,2,\ldots ,n-1,n)\right)^2\right|_{\x_a=\x_a^{(m)}}. 
	\ee
	where the sum runs over all ${\cal N}_{k,n}$ solutions to the scattering equations denoted $\x_a^{(m)}$. 
	
	In order to present our results, it is useful to review the definition of the higher dimensional Catalan numbers $C_m^{(d)}$. As it turns out, these numbers satisfy a duality relation $C_m^{(d)} = C_d^{(m)}$. This motivated us to write their explicit form in a way that manifests the symmetry. Moreover, the conjecture of \cite{MKandPK} states that $m_n^{(k)}(\mathbb{I},\mathbb{I}) = C^{(k)}_{n-k}$ with
	\be
	C^{(k)}_{n-k} := \frac{\prod_{p=0}^{n-k}p!\prod_{q=0}^k q!}{\prod_{r=0}^{n-1}r!}.
	% \cat{k}{n-k} := \frac{\prod_{p=0}^{n-k}p!\prod_{q=0}^k q!}{\prod_{r=0}^{n-1}r!}.
	\ee
	
	\subsection{Explicit Results}
	
	We have performed extensive computations and in every case we have found that $m_n^{(k)}(\mathbb{I},\mathbb{I}) = \cat{k}{n-k}$ on planar kinematics. 
	
	For $k=2$ Cachazo, He, and Yuan (CHY) conjectured in 2013 that $m_n^{(2)}(\mathbb{I},\mathbb{I})$, defined in terms of a sum over solutions, evaluates to the $(n-2)^{\rm th}$ Catalan number, $C_{n-2}$. This is consistent with the more general conjecture since $C^{(2)}_{n-2}$ is indeed the standard $(n-2)^{\rm th}$ Catalan number. In their paper CHY provided strong evidence for their conjecture. In 2014 Dolan and Goddard proved that $m_n^{(2)}(\mathbb{I},\mathbb{I})$ on general kinematics agrees with the sum over planar Feynman diagrams in a cubic scalar theory. If the planar kinematics is approached as a limit of general kinematics then each planar Feynman diagram evaluates to one and their sum simply becomes the number of planar cubic trees with $n$ leaves which is well-known to be $C_{n-2}$.
	
	For $k=3$ we have evaluated $m_n^{(3)}(\mathbb{I},\mathbb{I})$ on planar kinematics by summing over the solutions corresponding to cyclic fixed points for all $n\le40$. 
	
	The computation for $k=3$ is very straightforward since all cyclic fixed points that are solutions to the scattering equations admit a frame and therefore a standard gauge fixing.
	
	Let us move to $k=4$ cases where for the first time frameless solutions are found for $n>8$. Clearly, $k=4, n<8$ cases do not have frameless solutions since they are dual to $k=2$ and $k=3$ cases. 
	
	Let us start the $k=4$ discussion with $n=8$. The first step is to determine the cyclic fixed points that are solutions to the scattering equations. There are a total for ten triples of integers $(m_1,m_2,m_3)$ that are inequivalent under the $SL(4)$ and torus action. Of these, two are not solutions to the scattering equation. More explicitly, one can check that if $q=\exp (2\pi i /8)$ then 
	$$ 1+q+q^4+q^5 = 0 \quad {\rm and} \quad 1+q^2+q^4+q^6 = 0 $$
	and therefore $(1,4,5)$ and $(2,4,6)$ are not solutions to the scattering equations. 
	
	The remaining eight cyclic fixed points are solutions. Seven of them admit a standard frame using particles $\{1,2,3,4,5\}$. The seven solutions are
	\be
	\{(1, 2, 3), (1, 2, 4), (1, 2, 6), (1, 3, 4), (1, 3, 5), (1, 3, 6), (1,
	4, 6)\} .
	\ee
	The evaluation of the contributions to the amplitude from these seven solutions is easily done using the standard gauge fixing and gives rise to $24008$.
	
	The solution corresponding to $(1,2,5)$ has a matrix representative of the form
	\be
	\left(
	\begin{array}{cccccccc}
		1 & 0 & 0 & 0 & 1 & 1 & 1 & 1 \\
		0 & 1 & 0 & 0 & 2 & 1 & 1 & 0 \\
		0 & 0 & 1 & 0 & 0 & 1 & \frac{1}{2} & 1 \\
		0 & 0 & 0 & 1 & 1 & 1 & \frac{1}{2} & \frac{1}{2} \\
	\end{array}
	\right)
	\ee
	which makes it clear that a frame with particles $\{1,2,3,4,5\}$ is not possible but one with particles $\{1,2,3,4,6\}$ is. Computing the contribution to the amplitude gives $16$. Combining the two results we obtain 
	\be
	m_8^{(4)}(\mathbb{I},\mathbb{I}) = 24\, 024
	\ee
	which agrees with the four-dimensional Catalan number $C_4^{(4)}$.
	
	Now we are ready to discuss the first example in which frameless solutions are found. This is the case of $(k,n)=(4,9)$.
	
	There are a total for fourteen inequivalent triples of integers $(m_1,m_2,m_3)$. In this case all fourteen produce solutions to the scattering equations. 
	
	There are twelve solutions that admit a frame and two frameless solutions. The ones that admit a frame are
	$$(1, 2, 3), (1, 2, 4), (1, 2, 5), (1, 2, 6), (1, 2, 7), (1, 3, 4),$$
	$$(1, 3, 5), (1, 3, 7), (1, 4, 5), (1, 4, 6), (1, 5, 7), (2, 4, 6).$$
	and their contribution to the amplitude is $14\, 965\, 237/9$.
	
	The frameless triples are $(1, 3, 6)$ and $(1, 4, 7)$. Defining $r=\exp (\pi i/9)$ a matrix representative for $(1, 3, 6)$ is given by 
	\be
	\left(
	\begin{array}{ccccccccc}
		1 & 0 & 0 & 0 & 1 & 1 & 1 & 1 & 1 \\
		0 & 1 & 0 & 0 & r^7 & 0 & 0 & r^4 & 0 \\
		0 & 0 & 1 & 0 & 0 & r^5 & 0 & 0 & r^2 \\
		0 & 0 & 0 & 1 & -1 & -1 & r^6 & -1 & -1 \\
	\end{array}
	\right).
	\ee
	An exhaustive search shows that all possible subsets of five particles give rise to vanishing generalized Vandermonde determinants. 
	
	Following the construction of general $SL(4)$ gauge fixings provided in Section \ref{general gauge} it is possible to find a valid one given by 
	$$(I_1,I_2,\ldots ,I_{15}) =(10, 13, 15, 16, 17, 18, 19, 20, 21, 22, 23, 24, 25, 26, 27). $$
	In order words, the determinant of the $12\times 12$ submatrix of the Jacobian matrix obtained from columns and rows in ${1, 2, 3, 4, 5, 6, 7, 8, 9, 11, 12, 14}$ is non zero. 
	
	Defining $s=\exp (\pi i/3)$, the second frameless solution has a matrix representative of the form
	\be
	\left(
	\begin{array}{ccccccccc}
		1 & 0 & 0 & 0 & 1 & 1 & 1 & 1 & 1 \\
		0 & 1 & 0 & 0 & -1 & 0 & 0 & -s & 0 \\
		0 & 0 & 1 & 0 & 0 & -1 & 0 & 0 & -s \\
		0 & 0 & 0 & 1 & s & s & s^2 & s & s \\
	\end{array}
	\right).
	\ee
	
	It turns out that the same gauge fixing that works for $(1, 3, 6)$ also works for $(1, 4, 7)$. The combined contribution to the amplitude is $1/9$.
	
	Adding the contributions from all fourteen solutions one finds
	\be
	m_9^{(4)}(\mathbb{I},\mathbb{I}) = 1\, 662\, 804
	\ee
	which agrees with $C^{(4)}_5$.
	
	In principle there is no obstacle against computing  $m_n^{(k)}(\mathbb{I},\mathbb{I})$ to arbitrarily high values of $k$ and $n$ except for the computationally intensive task of searching for valid $SL(k)$ gauge fixings for frameless solutions.
	
	It is important to mention that in our numerical study we have found that when $n$ is prime there are no frameless solutions and computations can be carried out to high values of $(k,n)$. We list the computations we have performed below. In every case, the results agree with the high-dimensional Catalan conjecture.
	
	Results are listed with computation time, in Mathematica.
	
	\begin{itemize}
		\item $k=3$: $n\le 40$.
		\item $k=4$: All $n\le 15$.  Additionally $n=23$ (493 seconds) and $n=29$ (1839 seconds).
		\item $k=5$: $n=10,11,12,13,14,17$.  (For $n=19$: 6953 seconds).
		\item $k=6$: $n=13$ (8007 seconds).
	\end{itemize}

	\section{Tropical Grassmannian Evaluation and the Global Schwinger Parametrization}\label{sec:tropGrass Eval}
	%\section{Tropical Grassmannian Evaluation}\label{sec:tropGrass Eval}
	
	In \cite{MKandPK} evidence for the conjecture that $m_n^{(k)}(\mathbb{I},\mathbb{I})$ evaluates to the k-dimensional Catalan number $\cat{k}{n-k}$ on planar kinematics was obtained by evaluating  $m_n^{(k)}(\mathbb{I},\mathbb{I})$ as a sum over generalized Feynman diagrams \cite{Borges:2019csl}. Generalized Feynman diagrams are the $k>2$ analog of the standard planar cubic Feynman diagrams used to evaluate $m_n^{(2)}(\mathbb{I},\mathbb{I})$. In fact, planar kinematics for $k=2$ was originally designed to make each Feynman diagram evaluate to one so that $m_n^{(2)}(\mathbb{I},\mathbb{I})$ counts the number of such diagrams which is known to be $C^{(2)}_{n-2}$. Unfortunately, generalized Feynman diagrams (GFD) do not all evaluate to one on planar kinematics. The reason is that while some GFD's only possess poles in the planar basis and evaluate to one, other GFD's have other planar poles which are linear combinations of elements in the basis and therefore evaluate to rational numbers. 
	
	As it turns out, $k=2$ (and via duality $k=n-2$) is the only case when the dimension of the planar basis coincides with the dimension of the space of kinematic invariants. The fact that individual GFD's evaluate to rational numbers makes the counting interpretation implausible.
	
	In this section, we rewrite the sum over GFD's in a way, using an integral which we call the Global Schwinger Parametrization, that leads to a decomposition in terms of objects that evaluate to positive integer numbers.  Each of the new objects combines the contribution of several GFD's.
	
	In order to explain the construction, let us start by recalling that standard Feynman diagrams contributions to an amplitude can be thought of as the Laplace transform of certain regions in the Billera-Holmes-Vogtmann (BHV)  space of trees \cite{BilleraL}, which is also the tropical Grassmannian ${\rm Trop}\, G(2,n)$ \cite{speyer2004tropical,HJJS}.
	
	When restricting to $m_n^{(2)}(\mathbb{I},\mathbb{I})$ only planar Feynman diagrams contribute which leads to the positive tropical Grassmannian ${\rm Trop}^+G(2,n)$ introduced by Speyer and Williams in \cite{speyer2005tropical}. 
	
	The restriction to planar objects is very important because it allows us to find regions in kinematic space where the Laplace transform which computes individual Feynman diagrams exist simultaneously for all planar diagrams. This is not the case without the planarity condition, e.g. for $n=4$ there are three Feynman diagrams, with values $1/s,1/t,1/u$. In order to express one of them as a Laplace tranform of the space of trees (i.e. in a Schwinger parametrization) one needs the relevant Mandelstam invariant to be positive. However, momentum conservation $s+t+u=0$ allows at most two invariants to be positive simultaneously. Restricting to two of the three diagrams is in fact equivalent to imposing planarity. 
	
	This means that we can hope to be able to perform the Laplace transform of the whole ${\rm Trop}^+G(2,n)$ as a single integral if the elements in the planar basis are chosen to be positive.   
	
	Generalized Feynman diagrams extend the same ideas identifying $m_n^{(k)}(\mathbb{I},\mathbb{I})$ with the Laplace transform of ${\rm Trop}^+G(k,n)$. So far in the literature the Laplace transform has been carried out diagram by diagram since for generic kinematics it provides a systematic way of evaluation \cite{Borges:2019csl,Cachazo:2019xjx}. However, as mentioned above this obscures the way they should be combined when evaluated on planar kinematics. 
	
	Here we proceed by writing a formula for the Laplace transform of ${\rm Trop}^+G(k,n)$ as a single integral which on general kinematics can be decomposed in terms of individual GFD's but when evaluated on planar kinematics it performs the combination of GFD we are looking for. Note that planar kinematics sits inside the region where we expect the Laplace transform to exist. 
	
	Consider the Laplace transform representation of a single GFD, ${\cal T}$ \cite{Borges:2019csl,Cachazo:2019xjx},
	\be
	I_{\cal T} = \int d\mu_{\cal T}\, {\rm exp}\left(\,-{\cal F}\right)
	\ee
	where $d\mu_{\cal T}$ represents a measure over the space of internal edge lengths of the diagrams in the array of Feynman diagrams defining ${\cal T}$ while 
	\be\label{fino}
	{\cal F} = -\sum_{J\subset [n]:|J|=k}\s_{J} d_J.
	\ee
	Here $\s_J$ is the generalized Mandelstam invariant with a short hand notation for the indices while $d_J$ is the metric on the array of Feynman diagrams\footnote{The precise definition of arrays of Feynman diagrams and GFD is not needed in this work and we refer the reader to \cite{Borges:2019csl} and \cite{Cachazo:2019xjx} for details.} that make ${\cal T}$. If we let $J=\{j_1,\ldots ,j_k\}$ then $d_J$ is a completely symmetric tensor satisfying that the rank-two tensor constructed by fixing any $k-2$ indices and letting the remaining two vary is a metric on a binary tree (\cite{HJJS}). This means that $d_J$ satisfies all the three-term tropical Plucker vectors relations. This means that they define a point in the Dressian $Dr(k,n)$ (see \cite{speyer2004tropical,HJJS}). Restricting to planar GFD's further imposes that the tensor $d_J$ defines a point in the positive Dressian which was recently proven to be equal to the positive tropical Grassmannian ${\rm Trop}^+G(k,n)$, concurrently in \cite{speyer2020positive,Arkani-Hamed:2020cig}. 
	
	Using the connection to ${\rm Trop}^+G(k,n)$ and the fact that the invariants $\s_J$ satisfy the generalized momentum conservation one can write $d_J$ as the tropical Plucker coordinates evaluated on certain regions of ${\rm Trop}^+G(k,n)$.
	
	This means that the Laplace transform of the whole ${\rm Trop}^+G(k,n)$ must be equivalent to the sum over all GFD integrals $I_{\cal T}$,
	\begin{eqnarray}\label{eq:tropical integral0}
		m_n^{(k)}(\mathbb{I},\mathbb{I}) = \sum_{\cal T}I_{\cal T} = \int d\mu_{{\rm Trop}^+G(k,n)}\, \exp \left(\,-{\cal F}\right)
	\end{eqnarray}
	where the measure depends on the coordinates chosen. 
	
	Luckily Speyer and Williams \cite{speyer2005tropical} provided a natural construction of ${\rm Trop}^+G(k,n)$ based on the well-known positive Grassmannian $G^+(k,n)$. 
	
	The Speyer-Williams construction starts with a Web diagram and provides a matrix representative of a point in $G^+(k,n)$ as the boundary matrix of the diagram using edge variables. The $k(n-k)$ edge variables vary in $\mathbb{R}^+$ and generate $G^+(k,n)$. Given a matrix representative, Speyer and Williams proceed to map it to a point in ${\rm Trop}^+G(k,n)$ by tropicalizing the maximal minors. In the tropical object, the new ``edge" variables are now in $\mathbb{R}$. This can be understood by recalling that the tropical map can be thought of as the limit of an exponential map in which $x\in \mathbb{R}^+$ goes to $\exp (\tilde x)$ with $\tilde x \in \mathbb{R}$.   
	This immediately leads to the following formula 
	\be\label{eq:tropical integral}
	m_n^{(k)}(\mathbb{I},\mathbb{I}) =  \frac{1}{({\rm Vol}(\mathbb{R}^+))^{n-1}}\int_{\mathbb{R}^{k(n-k)}} d^{k(n-k)}\tilde x\, \exp \left(\,-{\cal F}\right)
	\ee
	with ${\cal F}$ as in \eqref{fino} but with the metric $d_J$ replaced with the tropicalized Plucker minors written in terms of the variables $\tilde x$. The reason for dividing by the volume of the torus $(\mathbb{R}^+)^{n-1}$ is the fact that the tropicalization procedure makes the scales of each column in the $k\times n$ representation of a point in $G^+(k,n)$ redundant. In physics terms, the model has a $(\mathbb{R}^+)^{n-1}$ gauge invariance. It is important to note that one of the $n$ possible rescalings has been fixed already when the standard $GL(k)$ action on the $k\times n$ matrix representatives of $G^+(k,n)$ was fixed. Once the redundancies are fixed the integral is over $\mathbb{R}^{(k-1)(n-k-1)}$ as expected.
	
	Before illustrating the construction with examples it is important to mention that a realization of the biadjoint amplitude $m^{(k)}_n(\mathbb{I},\mathbb{I})$ can also be obtained as the limit when $\alpha'\to 0$ of a string-like integral \cite{Arkani-Hamed:2019mrd}. This connection makes the evaluation of the amplitude that of a volume of a region defined in terms of tropical inequalities. Our formula, which integrates over tropicalized functions \eqref{eq:tropical integral}, seems compatible with the formulation in \cite{Arkani-Hamed:2019mrd} which computes the volume of a region defined by tropical inequalities (see Claim 3 of \cite{Arkani-Hamed:2019mrd}). Also, generalized biadjoint amplitudes have been evaluated using cluster algebra techniques \cite{Drummond:2019qjk,Drummond:2020kqg,Drummond:2019cxm,Henke:2019hve} in which the notion of a volume can be assigned to each cluster, which either coincides with a GFD or provide a refinement of one.

	Now we can proceed to the two main examples were we have done explicit computations. 
	
	\subsection{Case I: $k=2$}

	A matrix representative of a (generic) point in $G(2,n)$ can be parametrized as 
	\be
	\C =\left(
	\begin{array}{ccccccc}
		t_1 & 0 & t_3 & t_4(1+x_1) & t_5(1+x_1+x_2) & \cdots & t_n(1+x_1+x_2+\ldots +x_{n-3}) \\
		0 & t_2 & t_3 & t_4     &    t_5  &  \cdots    &  t_n \\
	\end{array}
	\right).
	\ee
	This parametrization differs slightly from that used by Speyer and Williams but the results are the same.  
	
	Let us introduce notation for the tropicalization of a Plucker coordinate:
	\be
	p_{a,b} := \det \left(\C_a,\C_b\right) \longrightarrow \tp a,b \tp.
	\ee
	Let us present some examples by computing the minors that enter when we specialize to planar kinematics. The first set is given by $p_{a,a+1}$ minors: 
	\be\label{eq:tropPlucker1}
	\begin{array}{lcl}
		p_{1,2} = t_1t_2 & \longrightarrow & \tp 1,2 \tp = {\tilde t}_1+{\tilde t}_2, \\
		p_{2,3} =-t_2 t_3 & \longrightarrow & \tp 2,3 \tp = {\tilde t}_2+{\tilde t_3}, \\
		p_{3,4} =-t_3 t_4 x_1 & \longrightarrow & \tp 3,4 \tp = {\tilde t}_3+{\tilde t}_4+{\tilde x}_1, \\
		p_{4,5} =-t_4t_5x_2 & \longrightarrow & \tp 4,5 \tp = {\tilde t}_4+{\tilde t}_5+{\tilde x}_2, \\
		& \vdots &  \\
		p_{n-1,n} = -t_{n-1}t_{n}x_{n-3} & \longrightarrow & \tp n-1,n \tp = {\tilde t}_{n-1}+{\tilde t}_n+{\tilde x}_{n-3}, \\
		p_{n,1} =-t_n t_1 & \longrightarrow & \tp n,1 \tp = {\tilde t}_{n}+{\tilde t}_1.
	\end{array}
	\ee
	
	The second set is $p_{a,a+2}$ minors: 
	
	\be\label{eq:tropPlucker2}
	\begin{array}{lcl}
		p_{1,3} = t_1t_3 & \longrightarrow & \tp 1,3 \tp = {\tilde t}_{1}+{\tilde t}_3, \\
		p_{2,4} =-t_2t_4(1+x_1) & \longrightarrow & \tp 2,4 \tp = {\tilde t}_{2}+{\tilde t}_4+\min (0,{\tilde x}_1), \\
		p_{3,5} =-t_3t_5(x_1+x_2) & \longrightarrow & \tp 3,5 \tp = {\tilde t}_{3}+{\tilde t}_5+\min ({\tilde x}_1,{\tilde x}_2), \\
		p_{4,6} =-t_4t_6(x_2+x_3) & \longrightarrow & \tp 4,6 \tp = {\tilde t}_{4}+{\tilde t}_6+\min ({\tilde x}_2,{\tilde x}_3), \\
		& \vdots &  \\
		p_{n-1,1} =-t_{n-1}t_1  & \longrightarrow & \tp n-1,1 \tp = {\tilde t}_{n-1}+{\tilde t}_1, \\
		p_{n,2} =t_n t_2(1+x_1+x_2+\ldots + x_{n-3}) & \longrightarrow & \tp n,2 \tp = {\tilde t}_{n}+{\tilde t}_2+\min (0,{\tilde x}_1,{\tilde x}_2,\ldots ,{\tilde x}_{n-3}).
	\end{array}
	\ee
	
	The last ingredient is to write 
	\be\label{tropF}
	{\cal F} =-\sum_{a<b} s_{ab} d_{ab} = \sum_{a<b} s_{ab} \tp a,b\tp. 
	\ee
	The minus sign on the RHS is needed to match the definition of $d_{ab}$ as a metric on trees.

	Note that ${\cal F}$ is independent of all ${\tilde t}$'s. In fact, the ${\tilde t}$'s could be identified with the lengths of the leaves once the integral is separated into individual trees. In order to see the independence note that every tropical minor in \eqref{tropF} has the form
	\be
	\tp a,b\tp = {\tilde t}_{a}+{\tilde t}_b+ \ldots, 
	\ee
	where we have exhibited all ${\tilde t}$ dependence. Using that the kinematic invariants satisfy
	\be
	s_{aa}=0\quad {\rm and}\quad \sum_{b=1}^n s_{ab} = 0 \quad \forall\, a
	\ee
	it is simple to see that all ${\tilde t}$'s drop out of ${\cal F}$. This means that the integrals over ${\tilde t}$'s factor out and cancel with the volume factors in \eqref{eq:tropical integral}.  
	
	Now we are ready to write down the Laplace transform over the entire ${\rm Trop}^+G(2,n)$ as a single integral,
	\be\label{intBA}
	m_n^{(2)}(\mathbb{I},\mathbb{I}) = \prod_{a=1}^{n-3}\int_{-\infty}^\infty d{\tilde x}_a\, \exp{\left(\,-{\cal F}\right)}.
	\ee
	Here we have dropped ${\tilde t}$ terms in ${\cal F}$.

	The combinatorial geometric interpretation which underlies the evaluation of $m^{(2)}_5$ is depicted in Figure \ref{fig:troppotheightfunction}.
\begin{figure}[h!]
	\centering
	\includegraphics[width=0.75\linewidth]{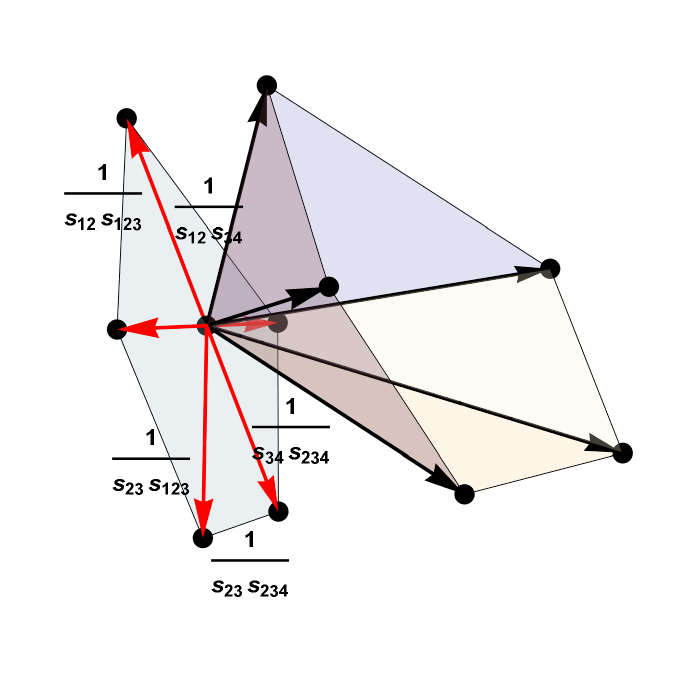}
	\caption{The Global Schwinger Parametrization, viewed as a projection of the positive tropical Grassmannian $\text{Trop}^+G(2,5)$; this is a projection of the root polytope. Integrating the piecewise-linear function $\mathcal{F}_5$, each of the five sectors contributes a single Feynman diagram. Rays are duality with planar kinematic invariants $s_{i\cdots j}$.}
	\label{fig:projectiontropgrass25}
\end{figure}
	
	Specializing to planar kinematics gives rise to 
	\be\label{notem}
	{\cal F}_{\rm PK} = \sum_{a=1}^{n-3}{\tilde x}_a -\min (0,{\tilde x}_1) - \sum_{a=1}^{n-4}\min ({\tilde x}_a,{\tilde x}_{a+1}) -  \min (0,{\tilde x}_1,{\tilde x}_2,\ldots ,{\tilde x}_{n-3}).  
	\ee
	Already the $k=2$ case is interesting because writing the amplitude on planar kinematics using \eqref{notem} requires a decomposition of the integration domain in \eqref{intBA} into regions where ${\cal F}_{\rm PK}$ becomes linear. The number of regions is much smaller than the number of standard Feynman diagrams. In fact, it coincides with the number of linear trees (see OEIS entry A045623,  \cite{oeis}) as we prove below. 
	
	\begin{prop}\label{linear trees}
		The number of regions needed to expand the piecewise function, ${\cal F}_{n,\rm PK}$ defined in \eqref{notem} is equal to the number of linear trees with $n$ leaves, i.e.
		$$ N_{n,\rm linear\, trees} = n \, 2^{n-5}. $$
	\end{prop}
	\begin{figure}[!h]%\label{fig:octahedron}
		\centering
		\includegraphics[width=1\linewidth]{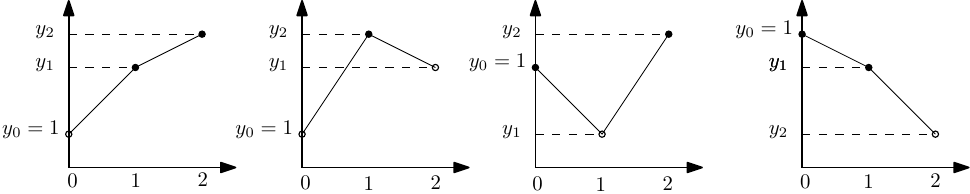}
		\caption{Counting argument used in the proof of Proposition \ref{linear trees}.}
	\end{figure}
	In order to prove the proposition and also to more easily evaluate the integral it is convenient to use 
	$$ \exp \left( \min (a_1,a_2,\ldots ,a_m) \right) = \min \left( \exp(a_1),\exp (a_2), \ldots ,\exp (a_m) \right), $$
	which follows from the fact that the exponential is a monotonically increasing function, in order rewrite the integral \eqref{notem}. Using a change of variables $y_a = \exp ({\tilde x}_a)$ one finds 
	\be\label{monra}
	m_n^{(2)}(\mathbb{I},\mathbb{I}) =  \prod_{a=1}^{n-3}\int_0^\infty \frac{dy_a}{y_a^2}\, \min (1,y_1)\prod_{a=1}^{n-4}\!\min(y_a,y_{a+1})\,\min(1,y_1,y_2,\ldots ,y_{n-3}).
	\ee
	
	Consider the first $n-3$ factors in the integrand and note that each one proves a choice between two options. For example, $\min(1,y_1)$ gives either $1<y_1$ or $1>y_1$. Therefore there are $2^{n-3}$ possibilities. Let $y_0:=1$. For any given one of the $2^{n-3}$ possibilities one can draw a mountain range picture by plotting the values of $\{ y_0,y_1,y_2,\ldots ,y_{n-3}\}$ in that order. Having constructed a mountain range it is easy to find out the number of options provided by the last factor in the integrand \eqref{monra}. The function $\min(1,y_1,y_2,\ldots ,y_{n-3})$ can only pick values from the valleys in the mountain range. Therefore the number of regions is given by 
	\be
	N_{n} = \sum_{i=1}^{2^{n-3}}V(R_i) 
	\ee
	where $V(R_i)$ is the number of valleys in the mountain range $R_i$. 
	
	Now let us find a recursion relations for $N_n$. Separate the ranges according to whether the first interval is up ($y_0<y_1$) or down ($y_0>y_1$). More explicitly, 
	\be
	N_n = \sum_{i=1}^{2^{n-4}}V(R_{i,D_1}) + \sum_{i=1}^{2^{n-4}}V(R_{i,U_1}),
	\ee
	where $R_{i,U_1}$ ($R_{i,U_1}$) are mountain ranges where the first interval is up $U_1$ (or down $D_1$).
	
	Clearly, if we go down then the remaining $n-4$ steps have the same number of valleys as they would if $y_0$ was removed. This means that the sum over them contributes $N_{n-1}$ to $N_{n}$. More explicitly,
	\be
	N_n = N_{n-1} + \sum_{i=1}^{2^{n-4}}V(R_{i,U_1}).
	\ee
	Next, we separate the ranges $R_{i,U_1}$ according to whether the second step is up or down. If the second step is down then one gets the contributions from $N_{n-2}$ plus one additional choice from the valley at $y_0$ for each of the graphs, i.e. a total of $2^{n-5}$. This gives 
	\be
	N_n = N_{n-1} + (N_{n-2}+2^{n-5})+\sum_{i=1}^{2^{n-5}}V(R_{i,U_1,U_2}).
	\ee
	Recursing the argument leads to
	\be
	N_n = N_{n-1}+\sum_{i=3}^{n-2}\left(N_i+2^{i-3}\right) +1 \quad {\rm with} \quad N_3=1.
	\ee
	Simplifying the recursion gives
	\be
	N_n = 2 N_{n - 1} + 2^{n - 5} \quad {\rm with} \quad N_3=1, \, N_4=5.
	\ee
	In this new form it is easy to the see that the solution agrees with the number of linear trees
	$$ N_n = N_{n,\rm linear\, trees} = n \, 2^{n-5} $$
	as expected.
	
	The first example in which there are trees that are not linear is $n=6$. In fact, there are exactly $2$ non-linear (snowflake) trees and $12$ linear ones. Computing the integral \eqref{monra}, i.e. 
	\be\nonumber
	m_6^{(2)}(\mathbb{I},\mathbb{I}) = \int_0^\infty \frac{dy_1}{y_1^2}\int_0^\infty \frac{dy_2}{y_2^2}\int_0^\infty \frac{dy_3}{y_3^2}\, \min (1,y_1)\min(y_1,y_2)\min(y_2,y_3)\min(1,y_1,y_2,y_3)
	\ee
	is a simple exercise once the integrand is separated into the $12$ regions, with $10$ regions evaluating to $1$ and two regions evaluating to $2$. Adding up gives 
	\be\nonumber
	m_6^{(2)}(\mathbb{I},\mathbb{I}) = 10\times 1 + 2\times 2 = 14
	\ee
	as expected.
	
	\subsection{Case II: $k=3$}
	
	Having seen that the ${\tilde t}$ scale factors drop from all computations, it is convenient to set them to one, i.e. ${\tilde t}_a=1$, from the start and use a matrix representative of a point in $G(3,n)$ of the form 
	\be
	\C =\left(
	\begin{array}{ccccccc}
		1 & 0 & 0 & 1 & 1+(1+y_1)x_1 & 1+(1+y_1)x_1+(1+y_1+y_2)x_2 & \cdots \\
		0 & 1 & 0 & 1 & 1+x_1 & 1+x_1+x_2 & \cdots  \\
		0 & 0 & 1 & 1 & 1      &    1  &  \cdots  \\
	\end{array}
	\right).
	\ee
	
	Once again this parametrization slightly differs from that used in \cite{speyer2005tropical} but the results are the same\footnote{Note that in the parameterization of the nonnegative Grassmannian, the entries in the second row of the matrix would usually come with minus signs; but for our purposes this is not necessary and we omit them.}.
	
	Let us present the minors that appear in ${\cal F}$ when evaluated on planar kinematics and their tropicalization. Once again the first set is given by $p_{a,a+1,a+2}$,
	\be\label{eq:tropPlucker3}
	\begin{array}{lcl}
		p_{1,2,3} = 1 & \longrightarrow & \tp 1,2,3 \tp = 0, \\
		p_{2,3,4} =1 & \longrightarrow & \tp 2,3,4 \tp = 0, \\
		p_{3,4,5} =-x_1y_1 & \longrightarrow & \tp 3,4,5 \tp = {\tilde x}_1+{\tilde y}_1, \\
		p_{4,5,6} =-x_1x_2y_2 & \longrightarrow & \tp 4,5,6 \tp = {\tilde x}_1+{\tilde x}_2+{\tilde y}_2, \\
		& \vdots &  \\
		p_{n-2,n-1,n} = -x_{n-5}x_{n-4}y_{n-4} & \longrightarrow & \tp n-2,n-1,n \tp = {\tilde x}_{n-5}+{\tilde x}_{n-4}+{\tilde y}_{n-4}, \\
		p_{n-1,n,1} =-x_{n-4} & \longrightarrow & \tp n-1,n,1 \tp = {\tilde x}_{n-4}, \\
		p_{n,1,2} =1 & \longrightarrow & \tp n,1,2 \tp = 0.
	\end{array}
	\ee
	
	The second set is given by minors of the form $p_{a,a+1,a+3}$,
	\be\label{eq:tropPlucker4}
	\begin{array}{lcl}
		p_{1,2,4} = 1 & \longrightarrow &  0, \\
		p_{2,3,5} =1+x_1+x_1y_1 & \longrightarrow & \min (0,{\tilde x}_1,{\tilde x}_1+{\tilde y}_1), \\
		p_{3,4,6} =-(x_1y_1+x_2y_1+x_2y_2) & \longrightarrow &  \min({\tilde x}_1+{\tilde y}_1,{\tilde x}_2+{\tilde y}_1,{\tilde x}_2+{\tilde y}_2), \\
		p_{4,5,7} =-x_1 (x_2 y_2+x_3 y_2+x_3 y_3) & \longrightarrow &  {\tilde x}_1+\min ({\tilde x}_2+{\tilde y}_2,{\tilde x}_3+{\tilde y}_2,{\tilde x}_3+{\tilde y}_3), \\
		& \vdots &\\
		p_{n-2,n-1,1} =y_{n-3-1}(1+x_1+\cdots +x_{n-3-1}) & \longrightarrow &  {\tilde y}_{n-k-1}+\min (0,\tilde{x}_1,\ldots, \tilde{x}_{n-3-1}), \\
		p_{n-1,n,2} =y_{n-3-1}(1+x_1+\cdots +x_{n-3-1}) & \longrightarrow &  {\tilde y}_{n-k-1}+\min (0,\tilde{x}_1,\ldots, \tilde{x}_{n-3-1}), \\
		p_{n,1,3} =-(1+y_1+\cdots +y_{n-3-1}) & \longrightarrow &  \min (0,\tilde{y_1},\ldots, \tilde{y}_{n-3-1}).\\
	\end{array}
	%\begin{array}{lcl}
	%p_{1,2,4} = 1 & \longrightarrow &  0, \\
	%p_{2,3,5} =1+x_1+x_1y_1 & \longrightarrow & \min (0,{\tilde x}_1,{\tilde x}_1+{\tilde y}_1), \\
	%p_{3,4,6} =-(x_1y_1+x_2y_1+x_2y_2) & \longrightarrow &  \min({\tilde x}_1+{\tilde y}_1,{\tilde x}_2+{\tilde y}_1,{\tilde x}_2+{\tilde y}_2), \\
	%p_{4,5,7} =-x_1 (x_2 y_2+x_3 y_2+x_3 y_3) & \longrightarrow &  {\tilde x}_1+\min ({\tilde x}_2+{\tilde y}_2,{\tilde x}_3+{\tilde y}_2,{\tilde x}_3+{\tilde y}_3), \\
	% & \vdots &\\
	% p_{n-2,n-1,1} =y_{n-k-1}(1+x_1+\cdots +x_{n-k-1}) & \longrightarrow &  {\tilde y}_{n-k-1}+\min (0,\tilde{x}_1,\ldots, \tilde{x}_{n-k-1}), \\
	%p_{n-1,n,2} =y_{n-k-1}(1+x_1+\cdots +x_{n-k-1}) & \longrightarrow &  {\tilde y}_{n-k-1}+\min (0,\tilde{x}_1,\ldots, \tilde{x}_{n-k-1}), \\
	%p_{n,1,3} =-(1+y_1+\cdots +y_{n-k-1}) & \longrightarrow &  \min (0,\tilde{y_1},\ldots, \tilde{y}_{n-k-1}).\\
	%\end{array}
	\ee
	
	See also Equation \eqref{PK minors web} for the general formula for the minors in the web parameterization.

	\subsubsection{Examples}
	
	Let us provide some example to show how specializing to planar kinematics before integrating over ${\rm Trop}^+G(3,n)$ gives rise to a different splitting into objects, each of which giving an integer contribution.
	
	The simplest case is $k=3$ and $n=6$. In order to simplify the notation we use $\{ x_a,y_a\}$ instead of $\{ {\tilde x}_a,  {\tilde y}_a \}$ for the integration variables. 
	
	The integral to be performed is
	\be
	\int_{\mathbb{R}^2} d^2x \int_{\mathbb{R}^2}d^2y\, \exp \left(-x_1- x_2-y_1-y_2+ G(x_1,x_2,y_1,y_2)\right),
	\ee
	with $G$ a piece-wise linear function
	\be
	\begin{array}{rcl}
		G(x_1,x_2,y_1,y_2) & :=  & \min \left(0,x_1,x_1+y_1\right)+\min
		\left(0,y_1,y_2\right)+ \\  &  & \min
		\left(x_1+y_1,x_2+y_1,x_2+y_2\right)+\min
		\left(0,x_1,x_2\right).
	\end{array}   
	\ee
	Here we have used that $\min
	\left(x_2,x_2+y_1,x_2+y_2\right) = x_2+\min
	\left(0,y_1,y_2\right)$.
	
	In the examples which follow, note that, as expected, the numbers of linear domains (regions) coincides with the numbers of facets in the respective root polytopes $\mathcal{R}_{k,n}$, as provided in the f-vectors listed in Example \ref{example: P36}).
	
	Separating the integration region into $27$ parts turns  $G(x_1,x_2,y_1,y_2)$ into a linear function in each. Note that if we had used generic kinematics, the corresponding piece-wise linear function would have required $48$ regions, i.e. the number of generalized Feynman diagrams. Evaluating the integral over the $27$ regions reveals three kinds of contributions. There are $16$ which contribute $1$, $10$ which contribute $2$ and a single one which contributes $6$ to the total integral. Combining the contributions gives rise to the value of the amplitude on planar kinematics
	\be
	m_6^{(3)}({\mathbb{I}},{\mathbb{I}}) = 16\times 1+10\times 2+1\times 6 = 42.
	\ee
	In order to express our results for $n=7$ and $n=8$ it is convenient to introduce a vector of values $v=(1,2,6)$ and one of the frequencies in which they appear, i.e. $f=(16,10,1)$ so that $f\cdot v=14$.
	
	For $k=3$ and $n=7$ we find $128$ regions and $10$ different values. The explicit results are:
	\be
	v = (1, 2, 3, 4, 5, 6, 8, 11, 12, 25), \quad f = (21, 38, 32, 8, 14, 2, 6, 3, 2, 2).
	\ee
	In this case $f\cdot v=462$ as expected. 
	
	We have also carried out the $n=8$ computation. There are $557$ regions and $36$ distinct values. The explicit results are:
	\be
	\begin{array}{rcl}
		v & = & (1, 2, 3, 4, 5, 6, 7, 8, 9, 10, 11, 12, 14, 15, 16, 17, 18, 20, 21, 24, 25, 26, 28, 30, 32, 33, 40, 42, 49, \\
		& & 54, 57, 75, 77, 93, 98, 169),\\
		f & = & (23, 42, 46, 57, 64, 47, 26, 18, 52, 26, 11, 20, 18, 6, 8, 10, 8, 10, 10, 4, 10, 6, 2, 7, 2, 2, 4, 2, 4, 2, \\ 
		& & 2, 2, 2, 1, 2, 1).
	\end{array}
	\ee
	Once again this leads to the expected result $f\cdot v=6\, 006$. 
	
	These decomposition of the integrals over ${\rm Trop}^+G(3,n)$ also provide a decomposition of the three-dimensional Catalan numbers. We leave the interpretation of this construction of $C^{(3)}_{n-3}$ for future work.

	Let us finally record the general formula for the integrand.  Define
	\begin{eqnarray}\label{PK minors web}
		P_i(\mathbf{x}) & = & \sum_{j=1}^{n-k} x_{i,j},\\
		Q_j(\mathbf{x}) & = & x_{1,j}x_{2,j}\cdots x_{k-1,j} + x_{1,j}x_{2,j}\cdots x_{k-2,j}x_{k-1,j+1}+x_{1,j}x_{2,j}\cdots x_{k-3,j}x_{k-2,j+1}x_{k-1,j+1}\nonumber\\
		& + & \cdots  +x_{1,j+1}x_{2,j+1}\cdots x_{k-1,j+1}.
	\end{eqnarray}
	\begin{claim}
		In the web parameterization, for the planar kinematics potential function we have
		\begin{eqnarray}\label{PK potential general form web}
			\mathcal{S}^{(PK)}_{k,n} =\log\left(\frac{\prod_{i=1}^{k-1}P_i(\mathbf{x})\prod_{j=1}^{n-k-1}Q_j(\mathbf{x})}{\prod_{(i,j)}x_{i,j}}\right),
		\end{eqnarray}
		where in the denominator $(i,j)$ ranges over the set 
		$$(i,j) \in \{1,\ldots, k-1 \} \times \{1,2,\ldots, n-k\}.$$
	\end{claim}
	
	We have checked this formula explicitly for nontrivial values of $(k,n)$, including $(k,n)\in \{(5,19),(6,18)\}$.
	
	Now setting all $x_{i,1} = 1$ for all $i=1,\ldots, k-1$ and then tropicalizing, we obtain the integrand of Equation \eqref{eq:tropical integral} specialized to planar kinematics.

	\section{Planar Scattering Equation in Terms of Cross-ratios and an Involution}\label{sec: PK scattering equations solutions are cyclic fixed points}
	
	In this section, we derive a projectively invariant expression for the $k=3$ and $k=4$ scattering equations in terms of cross-ratios which makes manifest the flip symmetry.  Flip symmetry in the kinematic space arises by the involution $j \mapsto n+1-j$, while the analogous action on the space of solutions of the PK scattering equations is by complex conjugation.

	It is not difficult to show that the planar kinematics scattering equations for $k=3$ have the  projectively invariant form
	\begin{eqnarray}\label{}
		\frac{p_{1,2,3}\, p_{2,4,5}\, p_{3,4,6}}{p_{1,2,4}\, p_{2,3,4}\, p_{3,5,6}}& = & 1\\
		\frac{p_{2,3,4}\, p_{3,5,6}\, p_{4,5,6}\, p_{4,5,7}}{p_{2,4,5}\, p_{3,4,5}\, p_{3,4,6}\, p_{5,6,7}} & = & 1,\nonumber
	\end{eqnarray}
	and the cyclic index permutations under the transformation $j\mapsto j+1$ modulo $n$, or equivalently
	that is
	\begin{eqnarray}\label{eqn:determinantal PK scattering}
		\frac{p_{1,2,3}\, p_{2,4,5}\, p_{3,4,6}}{p_{1,2,4}\, p_{2,3,4}\, p_{3,5,6}}& = & 1\\
		\frac{p_{1,2,4}\, p_{3,4,5}\, p_{3,5,6}}{p_{1,3,4}\, p_{2,3,5}\, p_{4,5,6}} & = & 1,\nonumber
	\end{eqnarray}
	together with all of equations obtained under cyclic index permutation.

	This form of the equations has the advantage that it makes manifest the fact that on PK the equations have a symmetry not shared by the definition of the kinematics. 
	
	% ------
	We define an involution on the kinematic space $\mathcal{K}(k,n)$ by $e^J \mapsto e^{J'}$, where $\{e^J: J\in \binom{\lbrack n\rbrack}{k}\}$ is the standard basis for $\mathbb{R}^{\binom{n}{k}}$, and where $J'$ is the flip of $J$: 
	$$\{j_1,\ldots, j_k\} \mapsto  \{n+1-j_1,\ldots, n+1-j_k\}.$$
	It follows that the new, flipped planar kinematics is now characterized by the equations
	$$\left\{\eta_J = 1: J\in \binom{\lbrack n\rbrack}{k}^{nf} \right\}\mapsto \left\{\eta_J = \ell-1: J\in \binom{\lbrack n\rbrack}{k}^{nf} \right\}$$
	% 	$$\eta_J = 1\ \ \mapsto  \ \ \eta_J = \ell-1,$$
	where $\ell$ is the number of cyclic intervals in $J$.
	
	Here $\binom{\lbrack n\rbrack}{k}^{nf}$ is the set of $k$-element subsets $J = \{j_1,\ldots, j_k\}$ that decompose the cycle $(1,2,\ldots, n)$ into at least two cyclic intervals.
	
	The effect for the coordinate functions $s_J$'s is to replace the conditions 
	$$\left\{ \s_{i,i+1,\ldots i+k-1,i+k+1}=-1: i=1,\ldots, n \right\}$$ by 
	$$\left\{\s_{i,i+2,\ldots i+k,i+k+1}=-1: i=-1,\ldots, n  \right\}.$$
	In other words, the invariants with a gap on the right are replaced by those with a gap on the left. Note that for $k=2$ there is no distinction between left and right and therefore the kinematics is invariant.
	
	Having defined the involution on the kinematic space, one can compute the new scattering equations associated to it. In the cross ratio form it is clear that the equations are invariant under the transformation. 
	
	Therefore, all solutions to the PK scattering equations are also solutions to the transformed version. This raises the question of whether there is an avatar of the involutive transformation on the solutions which maps them among themselves. 
	
	The way to find this out is the following. For any solution to the PK scattering equations defined by $\{ \w_1=q^{m_1},\w_2=q^{m_2}\}$ and $q=\exp (2\pi i/n)$, one has
	\be\label{direct}
	p_{a,a+1,a+3}=\det \left(
	\begin{array}{ccc}
		1 & 1 & 1 \\
		\w_1^a & \w_1^{a+1} & \w_1^{a+3} \\
		\w_2^a & \w_2^{a+1} & \w_2^{a+3}
	\end{array}
	\right) = (\w_1 \w_2)^a(1+\w_1+\w_2)(\w_1-1)(\w_1-\w_2)(\w_2-1)
	\ee
	and
	\be\label{reverse}
	p_{a,a+2,a+3}=\det \left(
	\begin{array}{ccc}
		1 & 1 & 1 \\
		\w_1^a & \w_1^{a+2} & \w_1^{a+3} \\
		\w_2^a & \w_2^{a+2} & \w_2^{a+3}
	\end{array}
	\right) = (\w_1 \w_2)^a(\w_1+\w_2+\w_1\w_2)(\w_1-1)(\w_1-\w_2)(\w_2-1).
	\ee
	Applying complex conjugation to $\w_1,\w_2$ one obtains another solution to the PK scattering equations with  $\{ \tilde \w\}_1,{\tilde \w}_2\} =\{ 1/\w_1,1/\w_2\} $.  Conjugating \eqref{direct} gives then
	\be
	p_{a,a+1,a+3}^* =  (\w_1 \w_2)^{(-2a-3)}p_{a,a+2,a+3}.
	\ee
	This means that one can define the action of the involution as conjugation, one can check that the set of all solutions remains invariant.
	
	In order to generalize the cross ratio form of the scattering equations to $k=4$ and beyond it is convenient to rewrite \eqref{eqn:determinantal PK scattering} in yet another form.  To this end, let us define a family of projective invariants, as follows\footnote{See Section 4.2 of \cite{Early:2019eun}.  One can show that $w^{(L)}_{i,j}$ is a monomial in the multi-split cross ratios $w_J$ for $J \in \binom{\lbrack n\rbrack}{k}^{nf}$.}.  Denote
	$$w^{(L)}_{i,j} = \frac{p_{L,i,j'} p_{L,i',j}}{p_{L,i,j}p_{L,i',j'}},$$
	where the set $L \cup \{i,i',j,j'\}$, with $L\in \binom{\lbrack n\rbrack}{k-2}$, has $k+2$ distinct elements, and where $i'$ and $j'$ are the immediate successors of respectively $i$ and $j$ in the standard cyclic order on $\{1,\ldots, n\} \setminus L$.
	
	Working backward we find that the $k=3$  planar basis kinematics scattering equations have the following expression in terms of cross-ratios: 
	\begin{eqnarray}\label{eq: k3 scattering cross ratios}
		w^{(2)}_{14} = w^{(3)}_{25},\ \ w^{(3)}_{14} = w^{(4)}_{25},
	\end{eqnarray}
	together with the set of cyclic shifts by $j\mapsto j+1\text{ mod}(n)$.  This gives a (redundant) system of $2n$ equations.  Here for instance
	$$w^{(2)}_{14} = \frac{p_{125}p_{234}}{p_{124}p_{235}},\ \ w^{(3)}_{14} = \frac{p_{135}p_{234}}{p_{134}p_{325}}.$$
	
	For an expression in terms of only minors made from one or two cyclic intervals, substituting $w^{(1)}_{14}\mapsto 1-w^{(a)}_{14}$ in Equation \eqref{eqn:determinantal PK scattering} gives the following set of equations: 
	\begin{eqnarray}
		\frac{p_{123}p_{245}}{p_{124}p_{235}} = \frac{p_{234}p_{356}}{p_{235}p_{346}},\ \ \frac{p_{124}p_{345}}{p_{134}p_{245}} = \frac{p_{235}p_{456}}{p_{245}p_{356}}.
	\end{eqnarray}
	
	Similarly, for $k=4$ we find
	\begin{eqnarray}\label{eq:cross ratio PK scattering equations}
		\frac{p_{1234} p_{2356} p_{3457}}{p_{1235} p_{2345} p_{3467}}=1,\ \ \frac{p_{1235} p_{2456} p_{3467}}{p_{1245} p_{2346} p_{3567}}=1,\ \ \frac{p_{1245} p_{3456} p_{3567}}{p_{1345} p_{2356} p_{4567}}=1
	\end{eqnarray}
	together with their cyclic shifts modulo $n$.

	These can be straightforwardly reorganized in terms of cross ratios, as
	\begin{eqnarray}\label{eq: k4 scattering cross ratios}
		w^{(23)}_{15} = w^{(34)}_{26},\ \ w^{(24)}_{15} = w^{(35)}_{26}, \ \ w^{(34)}_{15} = w^{(45)}_{26},
	\end{eqnarray}
	again together with the set of cyclic shifts by $j\mapsto j+1\mod(n)$.  Here for instance
	$$w^{(23)}_{15} = \frac{p_{1236}p_{2345}}{p_{1235}p_{2346}},\ \ w^{(24)}_{15} = \frac{p_{1246}p_{2345}}{p_{1245}p_{2346}},\ \ w^{(34)}_{15} = \frac{p_{1346}p_{2345}}{p_{1345}p_{2346}},$$
	which could be rewritten in terms of minors with two cyclic intervals by replacing $w^{(ab)}_{15}$ with $1-w^{(ab)}_{15}$.
	
	Unfortunately we could not achieve a systematic derivation starting from the scattering equations which would lead to a proof of a general cross-ratio formula for all $k$, but based on Equations \eqref{eq: k3 scattering cross ratios} and \eqref{eq: k4 scattering cross ratios} but it is natural to infer the following cross-ratio formulation of the PK scattering equations for any $k$ and $n$ (of course with $2\le k\le n-2$), as follows:
	\begin{eqnarray}\label{eqn:scattering higher rank cross ratio}
		&\left\{w^{(2,3,\ldots \widehat{j},\ldots,  k)}_{1,k+1} = w^{(3,4,\ldots \widehat{j+1},\ldots,  k+1)}_{2,k+2}:j=2,\ldots, k \right\}&\\
		&\left\{w^{(3,4,\ldots \widehat{j+1},\ldots,  k+1)}_{2,k+2}  = w^{(4,5,\ldots \widehat{j+2},\ldots,  k+2)}_{3,k+3}:j=3,\ldots, k+1 \right\},&\nonumber\\
		& \vdots &\nonumber
	\end{eqnarray}
	for a total of $(k-1)n$ (dependent) equations.
	
	\section{Polytopes: Roots and Deformations of the PK Point}\label{sec: blades, polytopes}
	In the rest of the paper our efforts are directed towards answering the natural question: are there good deformations of the PK point?  Usually $m^{(k)}(\mathbb{I}_n\mathbb{I}_n)$ is evaluated at generic kinematic points; on the other hand the PK point is extremely singular, with only 2n non-vanishing coordinates: $s_{i,i+1,\ldots, i+k-2,i+k}=-1$ and $s_{i,i+1,\ldots, i+k-2,i+k-1}=1$ with all others equal to zero.
	
	We summarize partial results towards answering this question.  Using the planar basis of functions on $\mathcal{K}(k,n)$, we embed a dimension $(k-1)(n-k-1)$ lattice polytope $\Pi_{k,n}$ in the kinematic space which has the following key property: it has the PK point as its unique interior lattice point.   Then we introduce the rank-graded root polytopes $\mathcal{R}_{k,n}$, which are related to $\Pi_{k,n}$ by duality, and we present evidence for our conjecture that the volume of $\mathcal{R}_{k,n}$ is the Catalan number modulo a factor intrinsic to the lattice,
	$$\text{Vol}(\mathcal{R}_{k,n}) = \frac{C^{(k)}_{n-k}}{((k-1)(n-k-1))!}.$$
	
	\subsection{Blades, Planar Bases and Polymatroidal Blade Arrangements}\label{sec: Blades, Planar Bases and Polymatroidal Blade Arrangements}

	In this section, for the reader's convenience we review some key results from earlier work, in particular \cite{Early:2019eun,Early:2020hap}; these culminate in Proposition \ref{prop:planar basis} and provide key ingredients for Definition \ref{defn:Pikn}, and they lie at the core of the embedding  $\Psi:\mathbb{R}^{\binom{n}{k}} \hookrightarrow \mathcal{K}(k,n)$ in Remark \ref{prop: PK point}, of the polytope $\Pi_{k,n}$ into the kinematic space.  
	
	Fix integers $(k,n)$ such that $1\le k\le n-1$.
	
	Recall the notation $\binom{\lbrack n\rbrack}{k}$ for the set of $k$-element subsets of the set $\lbrack n\rbrack = \{1,\ldots, n\}$, and denote by 
	$$\binom{\lbrack n\rbrack}{k}^{nf} = \binom{\lbrack n\rbrack}{k} \setminus \left\{\{j,j+1,\ldots ,j+k-1\}: j=1,\ldots, n \right\}$$
	the \textit{nonfrozen} $k$-element subsets.  Let $\{e^J: J\in \binom{\lbrack n\rbrack}{k}\}$ be the standard basis for $\mathbb{R}^{\binom{n}{k}}$.
	
	The $k^\text{th}$ hypersimplex in $n$ variables is the $k^\text{th}$ integer cross-section of the unit cube $\lbrack 0,1\rbrack^n$,
	$$\Delta_{k,n} = \left\{x\in \lbrack0,1 \rbrack^n: \sum x_j=k \right\}.$$
	Henceforth we shall assume that $2\le k\le n-2$.
	
	Recall that the \textit{lineality} space is the n-dimensional subspace
	$$\left\{\sum_{J}x_J e^J: x\in\mathbb{R}^n \right\},$$
	of $\mathbb{R}^{\binom{n}{k}}$, where we use the notation $x_J = \sum_{j\in J} x_j$.
	
	Then the \textit{kinematic space} is the dimension $\binom{n}{k}-n$ subspace of $\mathbb{R}^{\binom{n}{k}}$, 
	\begin{eqnarray}\label{eq: kinematic space}
		\mathcal{K}(k,n) = \left\{(s) \in \mathbb{R}^{\binomial{n}{k}}: \sum_{J:\ J\ni j}s_J=0,\ j=1,\ldots, n \right\}.
	\end{eqnarray}
	The original definition of blades is due to A. Ocneanu; blades were first studied in \cite{Early:2018mac} with connections to structures known in Hopf algebras as quasi-shuffles.  
	\begin{defn}[\cite{OcneanuLectures}]\label{defn:blade}
		A decorated ordered set partition $((S_1)_{s_1},\ldots, (S_\ell)_{s_\ell})$ of $(\{1,\ldots, n\},k)$ is an ordered set partition $(S_1,\ldots, S_\ell)$ of $\{1,\ldots, n\}$ together with an ordered list of integers $(s_1,\ldots, s_\ell)$ with $\sum_{j=1}^\ell s_j=k$.  It is said to be of type $\Delta_{k,n}$ if we have additionally $1\le s_j\le\vert S_j\vert-1 $, for each $j=1,\ldots, \ell$.  In this case we write $((S_1)_{s_1},\ldots, (S_\ell)_{s_\ell}) \in \text{OSP}(\Delta_{k,n})$, and we denote by $\lbrack (S_1)_{s_1},\ldots, (S_\ell)_{s_\ell}\rbrack$ the convex polyhedral cone in $\mathcal{H}_{k,n}$, that is cut out by the facet inequalities
		\begin{eqnarray}\label{eq:hypersimplexPlate}
			x_{S_1} & \ge & s_1 \nonumber\\
			x_{S_1\cup S_2} & \ge & s_1+s_2\nonumber\\
			& \vdots & \\
			x_{S_1\cup\cdots \cup S_{\ell-1}} & \ge & s_1+\cdots +s_{\ell-1}.\nonumber
		\end{eqnarray}
		These cones were called \textit{plates} by Ocneanu.
		Finally, the \textit{blade} $(((S_1)_{s_1},\ldots, (S_\ell)_{s_\ell}))$ is the union of the codimension one faces of the complete simplicial fan formed by the $\ell$ cyclic block rotations of $\lbrack (S_1)_{s_1},\ldots(S_\ell)_{s_\ell},\rbrack$, that is
		\begin{eqnarray}\label{eq: defn blade}
			(((S_1)_{s_1},\ldots, (S_\ell)_{s_\ell})) = \bigcup_{j=1}^\ell \partial\left(\lbrack (S_j)_{s_j},(S_{j+1})_{s_{j+1}},\ldots, (S_{j-1})_{s_{j-1}}\rbrack\right).
		\end{eqnarray}
	\end{defn}
	
	We emphasize that in this paper we consider only translations of the single nondegenerate blade with labeled by the cyclic order $(1,2,\ldots, n)$, usually denoted $\beta := ((1,2,\ldots, n))$; however in \cite{Early:2019zyi} it was shown that by pinning $\beta$ to a vertex $e_J$ of a hypersimplex $\Delta_{k,n}$, then that translated blade $\beta_{J}$ intersects the hypersimplex in a blade $(((S_1)_{s_1},\ldots, (S_\ell)_{s_\ell}))$ where now the pairs $(S_j,s_j)$ are uniquely determined and satisfy the condition from Definition \ref{defn:blade},
	$$1\le s_j\le \vert S_j\vert-1,$$
	or in short $((S_1)_{s_1},\ldots, (S_\ell)_{s_\ell}) \in \text{OSP}(\Delta_{k,n})$.  Additionally, we have that each $S_j = \{a,a+1,\ldots, b\}$ is cyclically contiguous.  We refer the reader to \cite{Early:2019zyi} for a detailed explanation of the construction of the decorated ordered set partition.
	
	The number of blocks $\ell$ is equal to the number of cyclic intervals in the set $J$ and the contents of the blocks are determined by the set $J$ together with the cyclic order.  In particular, the number of blocks is equal to the number of maximal cells in the subdivision induced by the blade.
	
	It was further shown in \cite{Early:2019zyi} that $(((S_1)_{s_1},\ldots, (S_\ell)_{s_\ell}))$ induces a multi-split positroidal subdivision where the vertices of the maximal cells become bases of Schubert matroids, or nested matroids.
	
	\begin{rem}
		We emphasize that a blade is not a tropical hyperplane, though the two are isomorphic as polyhedral complexes.  The blade $((\alpha_1,\ldots, \alpha_n))$ has the following key feature: the directions of its $n$ edges are exactly the cyclic system of roots $e_{\alpha_1}-e_{\alpha_2},\ldots, e_{\alpha_n}-e_{\alpha_1}$; this in fact has the nontrivial consequence that blades are tightly connected to the theory of matroids.  
		
		Superimposing multiple copies of the same blade $((1,2,\ldots, n))$ on the vertices of a hypersimplex $\Delta_{k,n}$ results in a particularly ``well-behaved'' subdivision when the vertices satisfy a condition on their pairwise relative displacements: in \cite{Early:2019zyi}, a combinatorial criterion called weak separation, for $k$-element subsets of $\{1,\ldots, n\}$, was shown to provide the compatibility criterion for an arrangement of blades on the vertices of $\Delta_{k,n}$ to induce a subdivision of it such that every maximal cell is a matroid (in particular positroid) polytope.  It is natural to ask what happens when the hypersimplex is replaced with more general classes of generalized permutohedra (or, polymatroids).  We return to this question at the end of the section.
		
	\end{rem}

	Let us now recall from \cite{Early:2019eun} the construction of the planar basis: this is a set of $\binomial{n}{k}-n$ linear functions, denoted $\eta_J$, on $\mathcal{K}(k,n)$ which are used to construct generalized Feynman diagrams in the sense of \cite{Borges:2019csl,Cachazo:2019xjx,Early:2019eun}.
	
	We first introduce $n$ linear functionals 
	$$L_j(x) = x_{j+1}+2x_{j+2}+\cdots +(n-1)x_{j-1},$$
	on $\mathbb{R}^n$, where the indices are cyclic modulo $n$.  For any lattice point $v\in\mathbb{R}^n$, define a piecewise-linear surface
	$$\rho_v(x) = -\frac{1}{n}\min\{L_1(x-v),\ldots, L_n(x-v)\}.$$
	We warn the reader that our convention differs from \cite{Early:2020hap} in that now the factor $-\frac{1}{n}$ is incorporated into $\rho_v$.
	
	We remark that unless otherwise stated we shall assume that $v$ has integer coordinates, so that $v\in\mathbb{Z}^n$.
	
	Here the graph of $\rho_v$ is a piecewise linear surface, with $n$ linear domains and with maximum height zero at $v$.  
	
	\begin{rem}
		The locus of nonzero curvature of the function $\rho_J$ is the blade $((1,2,\ldots, n))$, see Definition \ref{defn:blade}. 
	\end{rem}
	
	\begin{figure}[!h]\label{fig:octahedron}
		\centering
		\includegraphics[width=0.85\linewidth]{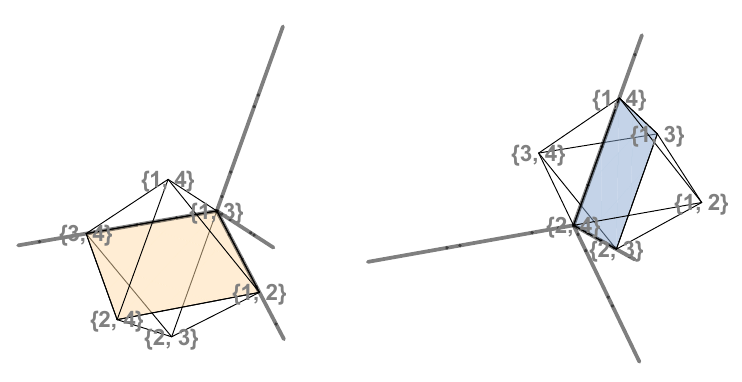}
		\caption{The two blade arrangements on the octahedron.  For clarity only portions of the blade are shown.  In physics these induce the $s$ and $t$ channels via the planar basis of linear functionals on the kinematic space.  Left: $\eta_{13} = s_{23}$.  Right: $\eta_{24} = s_{12}$.  For general $(k,n)$ see \cite{Early:2019zyi,Early:2019eun,Early:2020hap}.}
	\end{figure}

	The set of functions $\rho_u$ for $u\in \left\{x\in\mathbb{Z}^n: \sum_{j=1}^n x_j=0 \right\}$, say, satisfy relations which generalize the positive tropical Plucker relations from the hypersimplex to the ambient integer lattice.
	
	Indeed, a slight extension of Proposition\footnote{In the modification, we simply enlarge the vertex set of beyond that of $\Delta_{k,n}$, to other vertex sets an integer lattice of the form $\{x\in\mathbb{Z}^n: \sum_{J=1}^n x_j=r\}$ for a given integer $r$.} of \cite[Prop. 3.8]{Early:2020hap} leads to Proposition \ref{prop: octahedron recurrence}.
	\begin{prop}\label{prop: octahedron recurrence}
		Let $r\in\mathbb{Z}$ be an integer.
		
		For any $v\in \left\{x\in\mathbb{Z}^n: \sum_{j=1}^n x_j=r-2 \right\}$ and $x\in \left\{x\in\mathbb{Z}^n: \sum_{j=1}^n x_j=r \right\}$, then 
		\begin{eqnarray}\label{eq: octahedron recurrence}
			\rho_{v+e_{ac}}(x) + \rho_{v+e_{bd}}(x) = \min\{\rho_{v+e_{ab}}(x)+\rho_{v+e_{cd}}(x),\rho_{v+e_{ad}}(x)+\rho_{v+e_{bc}}(x) \}.
		\end{eqnarray}
		for any cyclic order $a<b<c<d$.
	\end{prop}
	
	\begin{proof}[Sketch of Proof]
		For the proof, the key insight is that around each integer lattice point, the hypersimplices $\Delta_{1,n},\Delta_{2,n},\ldots, \Delta_{n-1,n}$ meet, each with a multiplity $\binom{n}{k}$.  Therefore the proof reduces to the geometric one given in \cite{Early:2020hap}, for any chosen appropriated translated hypersimplex $\Delta_{2,n},\ldots, \Delta_{n-2,n}$.  For an analytic derivation, which we omit, one would show that   
		$$ - (\rho_{v+e_{ac}}(x)+\rho_{v+e_{bd}}(x)) + (\rho_{v+e_{ab}}(x)+\rho_{v+e_{cd}}(x)) $$
		and
		$$ - (\rho_{v+e_{ac}}(x)+\rho_{v+e_{bd}}(x)) + (\rho_{v+e_{ad}}(x)+\rho_{v+e_{bc}}(x)) $$
		have disjoint support on the given lattice $x\in \left\{x\in\mathbb{Z}^n: \sum_{j=1}^n x_j=r \right\}$, finding that when one of the two equations is nonzero at some $e^{v}$ for a lattice point $v$, then the coefficient is +1.
		
	\end{proof}

	For the present purposes we shall specialize the discussion to vertices $v = e_J\in \Delta_{k,n}$.  For each $k$-element subset $J \subset \{1,\ldots, n\}$, define a piecewise linear surface over the hypersimplex, the graph of the function $\rho_J:\Delta_{k,n} \rightarrow \mathbb{R}$, by
	$$\rho_J(x) = \min\left\{L_1(x-e_J),\ldots, L_n(x-e_J) \right\}.$$
	
	Often -- as is the case here -- it is important to localize the function $\rho_J$ still further to the vertices of a lattice polytope; for example, to the vertex set of $\Delta_{k,n}$.  We obtain a vector of heights, that is, an element of $\mathbb{R}^{\binom{n}{k}}$ with simple rational coefficients.
	
	Denote the localization of $\rho_J$ to the vertices of the hypersimplex by 
	\begin{eqnarray}
		\mathfrak{h}_J & = & \sum_{I\in\binom{\lbrack n\rbrack}{k}} \rho_J(e_I) e^I\\
		& = & \frac{1}{n}\sum_{I\in \binom{\lbrack n\rbrack}{k}}\min\{L_1(e_I-e_J),\ldots, L_n(e_I-e_J)\}e^I,
	\end{eqnarray}
	where $\{e^I: I \in \binom{\lbrack n\rbrack}{k} \}$ is the standard basis for $\mathbb{R}^{\binom{n}{k}}$. Then using the height function $\mathfrak{h}_J$, one can interpolate to recover the piecewise linear surface (i.e., the graph of $\rho_J$) over $\Delta_{k,n}$ whose bends project down to the blade $\beta_J$, which intersects the hypersimplex in a blade of the form 
	$$(((S_1)_{s_1},\ldots, (S_\ell)_{s_\ell})).$$
	
	Then each element $\mathfrak{h}_J$ determines a lift of each vertex $e_I \in \Delta_{k,n}$ to a height $\frac{1}{n}\min\{L_1(e_I-e_J),\ldots, L_n(e_I-e_J)\}$ and interpolating between these lifts gives rise to the piecewise-linear surface over the hypersimplex $\Delta_{k,n}$ defined by $\rho_J$, pinned to the vertex $e_J\in \Delta_{k,n}$.  
	
	Moreover, the bends of each such surface project down into $\Delta_{k,n}$ to the internal facets of a certain kind of matroid subdivision, called a positroidal multi-split (\cite{Early:2019zyi}).  This union of internal facets was further shown to coincide with the intersection of the translated blade $((1,2,\ldots, n))_{e_J}$ with the hypersimplex.

	\begin{example}
		Let us now give the explicit calculation of the identity of Proposition \ref{prop: octahedron recurrence}; of course, the basic example is the octahedron $\Delta_{2,4}$ itself.  The six height functions are as follows:
		\begin{eqnarray}
			\mathfrak{h}_{13} & = & -\frac{1}{4} \left(e^{12}+3 e^{14}+3 e^{23}+2 e^{24}+e^{34}\right)\nonumber\\
			\mathfrak{h}_{24} & = & -\frac{1}{4} \left(3 e^{12}+2 e^{13}+e^{14}+e^{23}+3 e^{34}\right)\nonumber\\
			\mathfrak{h}_{12} & = & -\frac{1}{4} \left(3 e^{13}+2 e^{14}+2 e^{23}+e^{24}+4 e^{34}\right)\\
			\mathfrak{h}_{34} & = & -\frac{1}{4} \left(4 e^{12}+3 e^{13}+2 e^{14}+2 e^{23}+e^{24}\right)\nonumber\\
			\mathfrak{h}_{14} & = & -\frac{1}{4} \left(2 e^{12}+e^{13}+4 e^{23}+3 e^{24}+2 e^{34}\right)\nonumber\\
			\mathfrak{h}_{23} & = & -\frac{1}{4} \left(2 e^{12}+e^{13}+4 e^{14}+3 e^{24}+2 e^{34}\right).\nonumber
		\end{eqnarray}
		Now we find that:
		\begin{eqnarray}
			(\mathfrak{h}_{12} + \mathfrak{h}_{34}) - (\mathfrak{h}_{13} + \mathfrak{h}_{24}) & = & e^{13}\\
			(\mathfrak{h}_{14} + \mathfrak{h}_{23}) - (\mathfrak{h}_{13} + \mathfrak{h}_{24}) & = & e^{24}.
		\end{eqnarray}
		Comparing respective coefficients we see that the height functions $e^{13}$ and $e^{24}$ above now obviously have disjoint support; it follows that indeed,
		$$\rho_{13}(e^{ij}) + \rho_{24}(e^{ij}) = \min \left\{\rho_{12}(e^{ij}) + \rho_{34}(e^{ij}), \rho_{14}(e^{ij}) + \rho_{23}(e^{ij})\right\}$$
		for each $e^{ij}$, where $e_{ij}$ is a vertex of the octahedron $\Delta_{2,4}$.
	\end{example}
	
	Now, as shown in \cite{Early:2020hap}, by specializing $\rho_v$ to the vertex set $\{e_J: J\in\binom{\lbrack n\rbrack}{k}\}$ of a hypersimplex, we find that the vectors $\mathfrak{h}_J\in \binom{\lbrack n\rbrack}{k}$ satisfy the positive tropical Plucker relations and thus define elements in the positive tropical Grassmannian Trop$^+G(k,n)$; in fact each $\mathfrak{h}_{J}$ generates a ray in Trop$^+G(k,n)$, and as a height function it induces the piecewise linear surface $\rho_J$ which projects down to the hypersimplex to induce a positroidal multisplit, such that $\rho_J$ is linear over each maximal cell.
	
	Consequently we obtain a family of piecewise-linear functions, pinned to the integer lattice points in an affine hyperplane of the form $\sum_{j=1}^n x_j=r\in\mathbb{Z}$ in $\mathbb{R}^n$, which can be localized to the integer lattice points in any generalized permutohedron; one particularly interesting case is when the facet hyperplanes are of the form $\sum_{j=a}^b x_j = c$ for any cyclic interval $a,a+1,\ldots, b$ and where $c$ is an integer.

	%%%%%%%%%%%%%%%%%%%%%%%%%%%%%%%%%%%%%%%%%%%%%%%%%%%%%%%%%%%

	Let us recall the first basis result result from \cite{Early:2020hap}.

	\begin{prop}[\cite{Early:2020hap}]\label{prop:planar basis height}
		The set of height functions $\mathfrak{h}_J$ is a basis for $\mathbb{R}^{\binom{n}{k}}$.
	\end{prop}
	
	Denote by ``$\cdot$'' the standard Euclidean dot product on $\mathbb{R}^{\binomial{n}{k}}$.
	
	Now for any $J\in \binom{\lbrack n\rbrack}{k}^{nf}$, define a linear functional on the kinematic space, or in more physical terminology, a planar kinematic invariant, $\eta_J:\mathcal{K}(k,n) \rightarrow \mathbb{R}$, by
	\begin{eqnarray}\label{eq:planar basis element}
		\eta_J(s) & := & -\mathfrak{h}_J \cdot (s) =-\frac{1}{n}\sum_{I\in \binom{\lbrack n\rbrack}{k}}\min\{L_1(e_I-e_J),\ldots, L_n(e_I-e_J)\}s_I.
	\end{eqnarray}
	Usually instead of $\eta_J(s)$ we write just $\eta_J$ with the understanding that $\eta_J$ is to be evaluated on points in $\mathcal{K}(k,n)$.
	
	Then we have the property that if $J = \{i,i+1,\ldots, i+k-1\}$ is frozen, then since the graph of $\rho_J:\Delta_{k,n}\rightarrow \mathbb{R}$ does not bend over $\Delta_{k,n}$, it follows that $\eta_J$ is identically zero on $\mathcal{K}(k,n)$.  See \cite{Early:2020hap} for details.

	A further computation proves linear independence for the set of $\eta_J$ where $J$ is nonfrozen, and we obtain Proposition \ref{prop:planar basis}.  This will be the key to defining the map $\Psi$ in Equation \eqref{eq:Psi}.
	
	\begin{prop}[\cite{Early:2020hap}]\label{prop:planar basis}
		The set of linear functionals
		$$\left\{\eta_J:\mathcal{K}(k,n) \rightarrow \mathbb{R}: J\in \binomial{\lbrack n\rbrack}{k}^{nf} \right\}$$
		is a basis of the space of the dual kinematic space $\left(\mathcal{K}(k,n)\right)^\ast$, that is to say, it is a basis of the space of linear functionals on $\mathcal{K}(k,n)$.
	\end{prop}
	
	For the conclusion of this section, we initiate the study of polymatroidal blade arrangements; these generalize the construction of matroidal blade arrangements in \cite{Early:2019zyi}.
	
	Note that in Definition \ref{defn:polymatroidal blade arrangement} we are including unbounded generalized permutohedra as maximal cells.
	\begin{defn}\label{defn:polymatroidal blade arrangement}
		Fix an integer $r\in\mathbb{Z}$.
		
		Given lattice points $v_1,\ldots, v_M$ in an affine hyperplane where $\sum_{j=1}^n x_j=r \in\mathbb{Z}$, call the arrangement of blades $\{\beta_{v_1},\ldots, \beta_{v_M}\}$ polymatroidal if every maximal cell in the superposition of the blades is a generalized permutohedron.
	\end{defn}
	Clearly, matroidal blade arrangements, where each $v_j$ is a vertex of a hypersimplex, provide a special case of this construction.
	
	\begin{cor}\label{cor: maximal cells polymatroidal blade arrangement}
		The maximal cells of a polymatroidal blade arrangement are generalized permutohedra\footnote{It is immediate that the bounded maximal cells occurring in a polymatroidal blade arrangement are polypositroids, as introduced very recently in \cite{2020LP}.} whose facets are in hyperplanes of the form $x_i+x_{i+1}+\cdots +x_{j} = c_{i,j}$.  
	\end{cor}

	\begin{example}
		Figure \ref{fig:hexagonblade} gives two polymatroidal blade arrangements on the vertices of a  two-dimensional generalized permutohedron.  Clearly, all of the maximal cells (some of which are unbounded) are generalized permutohedra (note that this included the unbounded case generalized permutohedra).
	\end{example}
	\begin{figure}[h!]
		\centering
		\includegraphics[width=0.6\linewidth]{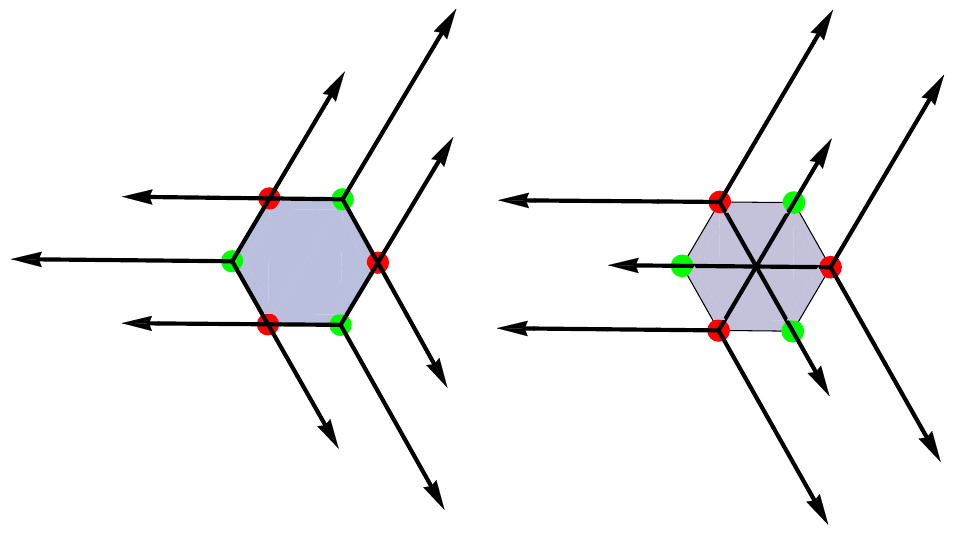}
		\caption{Two polymatroidal blade arrangements on the type $A_2$ root solid.}
		\label{fig:hexagonblade}
	\end{figure}
	
	\subsection{Polytopal Neighborhood of the Planar Kinematics Point}\label{sec: planar kinematics associahedron}\label{sec: Pikn}
	
	Recall that  $\{\alpha_{i,j}:(i,j)\in \lbrack 1,k-1\rbrack \times \lbrack 1,n-k\rbrack\}$ are coordinates on $\mathbb{R}^{(k-1)\times (n-k)}$.

	Let $(x_{i,j})_{\lbrack 1,k-1\rbrack \times \lbrack 1,n-k\rbrack}$ be auxiliary variables (appearing in the so-called web parameterization of the nonnegative Grassmannian). 
	% Here each of the zeroth variables $x_{i,0}$ is typically set to 1; in the calculation of the polytope $\Pi_{k,n}$ we find that homogeneous functions of the web variables are more useful.
	Define a codimension $k-1$ subspace $\mathcal{H}_{k,n}$ of $\mathbb{R}^{(k-1)\times (n-k)}$,
	$$\mathcal{H}_{k,n} = \left\{(\alpha_{ij}) \in \mathbb{R}^{(k-1)\times (n-k)}:  \sum_{j=1}^{n-k}\alpha_{i,j} = 0\text{ for each }i=1,\ldots, k-1\right\}.$$

	Define $\alpha_{i,\lbrack a,b\rbrack} = \sum_{j=a}^b \alpha_{i,j}$.  If $a>b$ then put $\alpha_{i,\lbrack a,b\rbrack}=0$.  More generally, put $\alpha_{I,J} = \sum_{(i,j)\in I\times J} \alpha_{i,j}$ for subsets $I \subseteq \{1,\ldots, k-1\}$ and $J \subseteq \{1,\ldots, k-n\}$.
	
	Definition \ref{defn:Pikn} contains the main construction of this section, of the lattice polytopal deformation of the PK point.
	\begin{defn}\label{defn:Pikn}
		Define a polyhedron $\Pi_{k,n} \subset \mathcal{H}_{k,n}$ by
		\begin{eqnarray}\label{eqn:PK subspace 0}
			\Pi_{k,n} = \left\{(\alpha_{ij})\in \mathcal{H}_{k,n} : \sum_{i=1}^{k-1}\alpha_{i,\lbrack j_i,j_{i+1}-i-1\rbrack}+1\ge 0,\ \  J \in \binom{\lbrack n\rbrack}{k}^{nf} \right\},
		\end{eqnarray}
		where $J = \{j_1,\ldots, j_k\}$ runs over all non-frozen subsets of $\{1,\ldots, n\}$.
	\end{defn}
	
	Then for instance if $J = \{1,4,5\}$ then correspondingly we have
	$$\alpha_{1,1} + \alpha_{1,2} +1\ge 0,$$
	while if $J = \{1,3,6\}$ then
	$$\alpha_{1,1} +\alpha_{2,23} +1 \ge 0.$$
	Note that $\alpha_{i,j} \ge -1$ is included in the set of inequalities; it follows that $\Pi_{k,n}$ is a bounded polyhedron with (at most) $\binom{n}{k}-n$ facets.

	\begin{prop}\label{prop: interior lattice point Pikn}
		For the polyhedron $\Pi_{k,n}$ we have the following two properties:
		\begin{enumerate}
			\item $\Pi_{k,n}$ has exactly $\binom{n}{k}-n$ facets.
			\item $\Pi_{k,n}$ has a unique interior lattice point $p_0$, given by $\alpha_{i,j} = 0$ for all $(i,j) \in \lbrack 1,k-1\rbrack \times \lbrack 1, n-k\rbrack$.
		\end{enumerate}
	\end{prop}

	\begin{proof}
		First note that the point $p_0$ with all coordinates $\alpha_{i,j}=0$, for $(i,j) \in \lbrack 1,k-1\rbrack \times \lbrack 1,n-k\rbrack$ satisfies all $\binom{n}{k}-n$ inequalities, but it does not minimize any of them.   Consequently $\Pi_{k,n}$ is nonempty and has the full dimension $(k-1)(n-k) - (k-1) = (k-1)(n-k-1)$.  
		
		For (1), we have that the $\binom{n}{k}-n$ facet inequalities are of the form
		$$ \sum_{i=1}^{k-1}\alpha_{i,\lbrack j_i,j_{i+1}-i-1\rbrack} \ge -1$$
		from which it is evident that they are additively independent and  consequently are minimized on distinct facets of $\Pi_{k,n}$.
		
		For (2) we finally claim that $p_0$ is the only interior lattice point.  Indeed, first note that  $\Pi_{k,n}$ lies inside the cube where $\lbrack -1, n-k-1\rbrack^{(k-1)(n-k)}$ and satisfies $\sum_{j=1}^{n-k}\alpha_{i,j}=0$ for each $i=1,\ldots, k-1$.  In particular, it lives in a Cartesian product of k-1 copies of the $(n-k)^{th}$ dilate of a simplex of dimension $n-k-1$, and each of these has exactly one interior lattice point at the origin.  Therefore the interior lattice point of $\Pi_{k,n}$ projects uniquely onto the origin in each copy. 
		
		The result follows.

	\end{proof}

	\begin{example}\label{example: PKpolytope 36}
		The polyhedron $\Pi_{3,6}$ is cut out by 14 facet inequalities in the codimension two subspace $\mathcal{H}_{3,6}$ of $\mathbb{R}^{(2)\times (3)}$ that is characterized by 
		$$\alpha_{1,123} = \alpha_{1,1} + \alpha_{1,2} + \alpha_{1,3} = 0,$$
		$$\alpha_{2,123} = \alpha_{2,1} + \alpha_{2,2} + \alpha_{2,3} = 0.$$   
		% $$\alpha_{1,123} = \alpha_{1,1} + \alpha_{1,2} + \alpha_{1,3} = 3,$$
		Now, $\alpha_{i,j} \ge -1$ accounts for $2\cdot 3=6$ facets.  The remaining $8$ facets minimize the following inequalities:
		\begin{eqnarray}
			\alpha_{i,1} + \alpha_{i,2} +1\ge  0,&\  \ \alpha_{i,2} + \alpha_{i,3} +1 \ge 0& \nonumber\\
			\alpha_{1,1} + \alpha_{2,2} +1 \ge0,& \  \ \alpha_{2,2} + \alpha_{2,3}  +1 \ge 0 &\\
			\alpha_{1,1} + \alpha_{1,2} + \alpha_{2,3} +1 \ge 0,& \  \ \alpha_{1,1} + \alpha_{2,2} + \alpha_{2,3} +1 \ge 0,& \nonumber 
		\end{eqnarray}
		where in the first line $i=1,2$.  Moreover, $\Pi_{k,n}$ has f-vector $(1,27,60,47,14,1)$; hence Euler characteristic zero.  It is interesting to note that the f-vector is the reverse of the one in Example \ref{example: P36}; we expect that this will be true in general for the two families of polytopes.
	\end{example}
	
	Recall from Proposition \ref{prop:planar basis} that the set of $\binom{n}{k}-n$ planar kinematic invariants $\eta_J(s):\mathcal{K}(k,n) \rightarrow \mathbb{R}$ is a basis for the space of linear functions on the kinematic space; we will use this property in the following construction to give an embedding of $\mathbb{R}^{(k-1)\times(n-k)}$ into $\mathcal{K}(k,n)$.
	
	For each $\alpha \in \mathbb{R}^{(k-1)\times(n-k)}$ we define a point in the kinematic space $s(\alpha)\in \mathcal{K}(k,n)$ by solving the system of equations 
	\begin{eqnarray}\label{eqn:PK subspace}
		\left\{\eta_J(s) = \sum_{i=1}^{k-1}\alpha_{i,\lbrack j_i,j_{i+1}-i-1\rbrack} +1,\ \ \eta_{i,i+1,\ldots, i+k-1}(s) = 0: J \in \binom{\lbrack n\rbrack}{k}^{nf},\ i=1,\ldots, n\right\}
	\end{eqnarray}
	for the coordinate functions $s_J$ on $\mathcal{K}(k,n)$.
	
	This gives rise to an embedding $\Psi:\mathbb{R}^{(k-1)\times(n-k)} \hookrightarrow \mathcal{K}(k,n)$,
	\begin{eqnarray}\label{eq:Psi}
		\Psi(\alpha) = s(\alpha),
	\end{eqnarray}
	which restricts to an embedding of $\Pi_{k,n}$ into a $(k-1)(n-k-1)$-dimensional subspace of the kinematic space.
	
	Now $\Pi_{k,n}$ has an interesting compatibility with planar kinematics, as in Proposition \ref{prop: PK point}.
	
	\begin{prop}\label{prop: PK point}
		We have that $\eta_J(\Psi(p_0)) = 1$ whenever $J\in\binom{\lbrack n\rbrack}{k}^{nf}$ and otherwise $\eta_J(\Psi(p_0)) = 0$ when $J$ is frozen, that is $\Psi(p_0) \in \mathcal{K}(k,n)$ is the planar kinematics point.
	\end{prop}
	
	In other words, the unique lattice point inside $\Pi_{k,n}$ \textit{is} the planar kinematics point!
	
	For instance, for the embedding $\Pi_{3,6} \hookrightarrow \mathcal{K}(3,6)$, we have
	$$\eta_{134} = \alpha_{1,1}+1,\ \eta_{245} = \alpha_{1,2}+1,\ \ \eta_{356} = \alpha_{1,3}+1,\ \eta_{145} = \alpha_{1,12}+1,\ \ \eta_{256} = \alpha_{1,23}+1,$$
	$$\eta_{124} = \alpha_{2,1}+1,\ \ \eta_{235} = \alpha_{2,2}+1,\ \ \eta_{346} = \alpha_{2,3}+1,\ \eta_{125} = \alpha_{2,12} + 1,\ \eta_{236} = \alpha_{2,23} + 1$$
	$$\eta_{135} = \alpha_{1,1} + \alpha_{2,2} + 1,\ \eta_{136} = \alpha_{1,1} + \alpha_{1,23} + 1,\ \eta_{146} = \alpha_{1,12} + \alpha_{2,3} + 1$$
	$$\ \eta_{246} = \alpha_{1,2} + \alpha_{2,3} + 1.$$
	In particular, the center $p_0=0\in \Pi_{k,n}$ is pushed to the PK point, where $\eta_J=1$.  Further, comparing with Example \ref{example: PKpolytope 36}, then the facet inequalities cutting out $\Pi_{3,6}$ become the exactly the conditions for the planar invariants $\eta_J$ to be nonnegative.
	
	We conclude this section with a proposal for the expression of $\Pi_{k,n}$ as a Newton polytope.
	
	\begin{claim}
		For any $2\le k\le n-2$, then the polyhedron $\Pi_{k,n}$ is equal to the Newton polytope of the Laurent polynomial appearing in Equation \eqref{PK potential general form web}:
		\begin{eqnarray}\label{eq:Laurent polynomial PK}
			\frac{\prod_{i=1}^{k-1}P_i(\mathbf{x})\prod_{j=1}^{n-k-1}Q_j(\mathbf{x})}{\prod_{(i,j)}x_{i,j}},
		\end{eqnarray}
		where
		\begin{eqnarray}\label{PK minors web2}
			P_i(\mathbf{x}) & = & \sum_{j=1}^{n-k} x_{i,j},\\
			Q_j(\mathbf{x}) & = & x_{1,j}x_{2,j}\cdots x_{k-1,j} + x_{1,j}x_{2,j}\cdots x_{k-2,j}x_{k-1,j+1}+x_{1,j}x_{2,j}\cdots x_{k-3,j}x_{k-2,j+1}x_{k-1,j+1}\nonumber\\
			& + & \cdots  +x_{1,j+1}x_{2,j+1}\cdots x_{k-1,j+1}.\nonumber
		\end{eqnarray}
		In fact, note that directly from Equation \eqref{PK minors web2}, by calculating total degrees of the variables $x_{i,j}$, it follows that the Newton polytope for Equation \eqref{eq:Laurent polynomial PK} is in the subspace $\mathcal{H}_{k,n}$ and in particular it contains the origin in its interior.  
	\end{claim}

	\subsection{Rank-Graded Root Polytopes and Their Volumes}\label{sec:PK polytope}
	
	In what follows, we initiate the study of a polytope which is in duality with $\Pi_{k,n}$, modulo a change of coordinates.  We first introduce a $(k-1)(n-k)$-dimensional polytope $\hat{\mathcal{R}}_{k,n}$, and its $(k-1)(n-k-1)$-dimensional projection $\mathcal{R}_{k,n}$.  Our primary aim is to investigate properties of the latter.  By abuse of terminology we may call both $\hat{\mathcal{R}}_{k,n}$ and its projection $\mathcal{R}_{k,n}$ root polytopes, but it will be clear from the context which one we mean.
	
	After providing computational evidence, at the end of the section we formulate Conjecture \ref{conjecture: PK volume}, that the volume of $\mathcal{R}_{k,n}$ is 
	$$\text{Vol}(\mathcal{R}_{k,n}) = \frac{C^{(k)}_{n-k}}{((k-1)(n-k-1))!}.$$
	
	On the other hand, a dimension $k(n-k)$ polytope analogous to $\mathcal{R}_{k,n}$, called the superpotential polytope $\Gamma_G$, was studied in \cite{rietsch2019newton}.  Rewriting the formula given in Proposition 16.8 of  \cite{rietsch2019newton}, the volume is 
	$$\text{Vol}(\Gamma_G) = \frac{\cat{k}{n-k}}{((k)(n-k))!}.$$

	In Example \ref{example: type A root polytope}, we check that $\mathcal{R}_{2,5}$ coincides with (a projection of) the so-called root polytope of type $A_{2}$.  The identification clearly extends to any $n$.  
	
	In particular, $\hat{\mathcal{R}}_{2,n}$ is the largest among the family of root polytopes introduced in \cite{postnikov2009permutohedra}; these are by construction the convex hull of the origin together with all positive roots $e_i-e_j$ with $i<j$.
	
	Recall that $\{e_{i,j}: 1\le i\le k-1,\ \ 1\le j\le n-k \}$ is the standard basis for $\mathbb{R}^{(k-1)\times (n-k)}$, and $\{\alpha_{i,j}: 1\le i\le k-1,\ \ 1\le j\le n-k \}$ is the set of coordinate functions.
	
	The polytope $\hat{\mathcal{R}}_{k,n}$ lives in the space 
	$$\hat{\mathcal{H}}_{k,n} = \left\{(\alpha_{ij}) \in \mathbb{R}^{(k-1)\times (n-k+1)}:  \sum_{j=1}^{n-k+1}\alpha_{i,j} =0\text{ for each }i=1,\ldots, k-1\right\}.$$
	Recall also the subspace of $\hat{\mathcal{H}}_{k,n}$,
	$$\mathcal{H}_{k,n} =  \left\{(\alpha_{ij}) \in \mathbb{R}^{(k-1)\times (n-k)}:  \sum_{j=1}^{n-k}\alpha_{i,j} =0\text{ for each }i=1,\ldots, k-1\right\}$$
	and denote by $\text{proj}_{k,n}: \hat{\mathcal{H}}_{k,n}\rightarrow \mathcal{H}_{k,n}$ the projection determined by $\text{proj}_{k,n}(e_{i,j}) = e_{i,j}$ for $j\le n-k$, and 
	$$\text{proj}_{k,n}(e_{i,n-k+1}) = e_{i,1}.$$
	
	Finally, for each nonfrozen $J = \{j_1,\ldots, j_k\} \in \binom{\lbrack n\rbrack}{k}^{nf}$, let
	\begin{eqnarray}\label{eqn: vertices Pk root}
		\hat{v}_J & = & \sum_{\lambda=1}^{k-1}(e_{\lambda,j_{\lambda}-(\lambda-1)} - e_{\lambda,j_{\lambda+1}-(\lambda-1)-1}).
	\end{eqnarray}
	\begin{defn}
		The polytope $\hat{\mathcal{R}}_{k,n} \subset \hat{\mathcal{H}}_{k,n}$ is the convex hull of the origin, together with the following $\binom{n}{k}-n$ points: 
		$$\left\{\hat{v}_J: J\in \binom{\lbrack n\rbrack}{k}^{nf}\right\},$$
		as well as the $k-1$ points
		$$\left\{e_{i,1} - e_{i,n-k+1}: i=1,\ldots, k-1 \right\}.$$
	\end{defn}
	
	Now define 
	$$	v_J = \sum_{\lambda=1}^{k-1}(e_{\lambda,j_{\lambda}-(\lambda-1)} - e_{\lambda,j_{\lambda+1}-(\lambda-1)-1}),
	$$
	where now the subscripts are now by convention taken modulo $n-k$.
	
	\begin{defn}
		The polytope $\mathcal{R}_{k,n} \subset \mathcal{H}_{k,n}$ is the convex hull of the following $\binom{n}{k}-n$ points: 
		$$\left\{v_J: J\in \binom{\lbrack n\rbrack}{k}^{nf}\right\}.$$
	\end{defn}
	Note that the origin is already in the convex hull.
	
	Observe that
	$$\text{proj}_{k,n}(e_{i,1} - e_{i,n-k+1}) = 0$$
	for all $i=1,\ldots, k-1$, and consequently $\text{proj}_{k,n}(\hat{\mathcal{R}}_{k,n}) = \mathcal{R}_{k,n}$.
	
	\begin{rem}\label{prop: Pikn reflexive}
		We claim that the polytope $\mathcal{R}_{k,n}$ introduced in the Section \ref{sec:PK polytope} is in duality with the polyhedron $\Pi_{k,n}$ defined in Section \ref{sec: Pikn}.  This can be seen as follows.  Define new elements
		$$f_{i,j} = e_{i,j} - e_{i,j+1}.$$
		Then the vertices of $\mathcal{R}_{k,n}$ take the form
		$$v_J = \sum_{i=1}^{k-1} f_{i,\lbrack j_i,j_{i+1}-i-1\rbrack},$$
		which dualizes to the linear function defining the facet hyperplane
		$$\sum_{i=1}^{k-1} \alpha_{i,\lbrack j_i,j_{i+1}-i-1\rbrack}+ 1 = 0,$$
		of $\Pi_{k,n}$.
	\end{rem}
	
	\begin{example}\label{example: type A root polytope}
		Consider the type $A_2$ root polytope.  In the present convention this is equal to 
		$$\hat{\mathcal{R}}_{2,5} = \text{convex hull} \{0,e_{1,1}-e_{1,2},e_{1,1}-e_{1,3},e_{1,1}-e_{1,4},e_{1,2}-e_{1,3},e_{1,2}-e_{1,4},e_{1,3}-e_{1,4}\} \subset \hat{\mathcal{H}}_{2,5}.$$
		Now its projection is
		$$\mathcal{R}_{2,5} = \text{proj}_{2,5}(\hat{\mathcal{R}}_{k,n})= \text{convex hull} \{e_{1,1}-e_{1,2},e_{1,1}-e_{1,3},e_{1,2}-e_{1,3},e_{1,2}-e_{1,1},e_{1,3}-e_{1,1}\} \subset \mathcal{H}_{2,5},$$
		that is,
		$$\mathcal{R}_{2,5} = \text{convex hull}\left\{v_{1,3},v_{1,4},v_{2,4},v_{2,5},v_{3,5} \right\}.$$
	\end{example}

	\begin{example}\label{example: P36}
		Let us also present $\mathcal{R}_{3,6}$ explicitly.  This is the convex hull in $\mathcal{H}_{3,6}$ of the following 14 points:
		\begin{eqnarray}
			\begin{array}{cc}
				v_{1,2,4} & e_{2,1}-e_{2,2} \\
				v_{1,2,5} & e_{2,1}-e_{2,3} \\
				v_{1,3,4} & e_{1,1}-e_{1,2} \\
				v_{1,3,5} & e_{1,1}-e_{1,2}+e_{2,2}-e_{2,3} \\
				v_{1,3,6} & e_{1,1}-e_{1,2}-e_{2,1}+e_{2,2} \\
				v_{1,4,5} & e_{1,1}-e_{1,3} \\
				v_{1,4,6} & e_{1,1}-e_{1,3}-e_{2,1}+e_{2,3} \\
				v_{2,3,5} & e_{2,2}-e_{2,3} \\
				v_{2,3,6} & e_{2,2}-e_{2,1} \\
				v_{2,4,5} & e_{1,2}-e_{1,3} \\
				v_{2,4,6} & e_{1,2}-e_{1,3}-e_{2,1}+e_{2,3} \\
				v_{2,5,6} & e_{1,2}-e_{1,1} \\
				v_{3,4,6} & e_{2,3}-e_{2,1} \\
				v_{3,5,6} & e_{1,3}-e_{1,1}.
			\end{array}
		\end{eqnarray}
		Using a computer program such as SageMath, one finds that $\mathcal{R}_{3,6}$ has f-vector given by
		$$(1,14,47,60,27,1),$$ and volume
		$$\text{Volume}(\mathcal{R}_{3,6}) = \frac{42}{4!} = \frac{\cat{3}{3}}{4!},$$
		where $\cat{k}{n-k}$ is the multi-dimensional Catalan number.
	\end{example}
	
	Similarly we find that $\mathcal{R}_{3,7}$, ($\mathcal{R}_{3,8}$ and $\mathcal{R}_{5,8}$), $\mathcal{R}_{3,9}$ and $\mathcal{R}_{3,11}$ have f-vectors respectively 
	\begin{eqnarray*}
		&(1, 28, 178, 483, 661, 456, 128, 1),&\\
		&(1,48,486,2122,5030,7048,5895,2750,557,1),&\\
		&(1,75,1108,6948,24170,52281,73891,68921,41244,14474,2286,1),&\\
		&(1, 154, 4179, 45769, 278224, 1081720, 2898751, 5583293, 7902473, 8280735, 6383651,&\\
		& 3537888, 1341425, 313380, 34236, 1)&   
	\end{eqnarray*}
	while $\mathcal{R}_{4,8}$, $\mathcal{R}_{4,9}$ have f-vectors
	\begin{eqnarray*}
		&(1, 62, 770, 4048, 11653, 20409, 22559, 15524, 6133, 1074, 1),&\\
		&(1, 117, 2441, 20488, 94620, 275905, 544210, 750799, 731318, 496454, 225059, 61668, 7783, 1).&\\
	\end{eqnarray*}
	Moreover, as in Example \ref{example: P36} we find that the volumes are the fractions
	$$\text{Vol}(\mathcal{R}_{k,n}) = \frac{\cat{k}{n-k}}{((k-1)(n-k-1))!},$$
	for $\mathcal{R}_{3,n}$ with $n\le 9$ and $\mathcal{R}_{4,n}$ with $n\le 9$.  In particular the \textit{relative} volume is the multi-dimensional Catalan number $\cat{k}{n-k}$ itself.
	
	Using SageMath we were also able to compute the f-vectors of $\mathcal{R}_{4,10}$ and $\mathcal{R}_{5,10}$ (in about ten hours), respectively
	\begin{eqnarray*}
		&& (1, 200, 6463, 79151, 525529, 2217016, 6460534, 13639822, 21436558, 25407704, 22742748,\\
		&&15211454, 7404964, 2490478, 520155, 51128, 1)
	\end{eqnarray*}
	and
	\begin{eqnarray*}
		&& (1, 242, 9041, 123808, 907951, 4218658, 13571560, 31822956, 56070720, 75497722, 78187219,\\
		&&62086930, 37284006, 16453106, 5055558, 970826, 88193, 1).
	\end{eqnarray*}

	Now we arrive at the initial raison d'etre for rank-graded root polytopes: their (relative) volume computes the generalized biadjoint scalar $m^{(k)}(\mathbb{I}_n,\mathbb{I}_n)$ at the planar kinematic point!  Indeed, this demonstrates a compatibility with the discussion in Section 2.2 of \cite{Arkani-Hamed:2019mrd} concerning the volume of the dual polytope. 
	
	Moreover, we see directly that the number of facets of $\mathcal{R}_{3,6}, \mathcal{R}_{3,7}$ and $\mathcal{R}_{3,8}$ in the f-vectors listed above coincide with the number of linear domains used in Section \ref{sec:tropGrass Eval} to evaluate $m^{(k)}(\mathbb{I}_n,\mathbb{I}_n)$ as a Laplace type transform at the planar kinematics point.
	
	For a small preview of the structure of rank graded root polytope $\hat{\mathcal{R}}_{k,n}$ and its projection $\mathcal{R}_{k,n}$, we observe a basic feature of their face posets which can be easily verified, by simply expanding the vertices $v_J$ in the standard basis of $e_{i,j}$'s as in Equation \eqref{eqn: vertices Pk root}.
	\begin{prop}\label{prop: graded embeddings root polytopes}
		Whenever $2\le k-1<k\le n-2$ then we have $k-1$ natural embeddings
		$$\hat{\mathcal{R}}_{k-1,n-1} \simeq \hat{\mathcal{R}}_{k,n} \cap \left\{\alpha \in \hat{\mathcal{H}}_{k,n}: \alpha_{i,j} = 0\text{ for all } j=1,\ldots, n-k \right\},$$
		one for each $i=1,\ldots, k-1$, and similarly for $\mathcal{R}_{k,n}$.
	\end{prop}
	\begin{proof}
		The subpolytopes can be constructed explicitly.  For instance, the vertices of the polytope $\mathcal{R}_{k-1,n-1}$ in the $i^\text{th}$ embedding into $\mathcal{R}_{k,n}$, that is such that $\alpha_{i,j} = 0$ for all $j=1,\ldots, n-k$, are those vertices $v_J \in \mathcal{R}_{k,n}$ where $J\in \binom{\lbrack n\rbrack}{k}^{nf}$ has the property that $j_i+1 = j_{i+1}$.  For $\hat{\mathcal{R}}_{k,n}$ the procedure it exactly analogous.
	\end{proof}
	
	\begin{conjecture}\label{conjecture: PK volume}
		The polytope $\mathcal{R}_{k,n}$ has (relative) volume the $k$-dimensional Catalan number $\cat{k}{n-k}$. 
		
		In particular, 
		$$\text{Vol}(\mathcal{R}_{k,n}) = \frac{\cat{k}{n-k}}{((k-1)(n-k-1))!}.$$
	\end{conjecture}
	A combinatorial proof of Conjecture \ref{conjecture: PK volume} would be very interesting, see \cite{CE2020}.
	
	\begin{rem}
		Finally we observe that a similar construction to that used for Proposition \ref{prop: PK point} shows that $\mathcal{R}_{k,n}$ can be embedded in the kinematic space as a minimal polytopal neighorhood of the PK point.
	\end{rem}
	
	\subsection{Conical Kinematics: Roots and Weights}
	
	In this section, we introduce conical kinematics, which provides a generalization of the construction for k=2 in \cite{Early:2018zuw}, which in particular gives a different value for $m^{(k)}(\mathbb{I}_n,\mathbb{I}_n)$ for $k\ge 3$ from the minimal kinematics in \cite{MKandPK}.
	
	For any $J\in \binom{\lbrack n\rbrack}{k}^{nf}$, define the linear functional
	\begin{eqnarray}\label{eq: gamma def}
		\gamma_J = \sum_{\ell=1}^{k-1}(\alpha_{\ell,j_{\ell}-(\ell-1)} - \alpha_{\ell,j_{\ell+1}-(\ell-1)-1}),
	\end{eqnarray}
	where the indices satisfy $(i,j)\in \lbrack 1,k-1\rbrack \times \lbrack 1,n-k+1\rbrack$.
	
	\begin{conjecture}\label{conjecture: conical kinematics amp}
		For any $\alpha\in \mathbb{R}^{(k-1)(n-k+1)}$, solve the equations $\eta_J(s) = \gamma_J$ for the coordinate functions $s$ on $\mathcal{K}(k,n)$.  Then the scattering equations possess a unique solution, and we have 
		$$m^{(k)}(\mathbb{I}_n,\mathbb{I}_n)\big\vert_{\eta_J = \gamma_J} = \prod_{i=1}^{k-1} \left(\frac{\alpha_{i,1} - \alpha_{i,n-k+1}}{\prod_{j=1}^{n-k}\left(\alpha_{i,j} - \alpha_{i,j+1}\right)}\right).$$
	\end{conjecture}
	
	There is a second version of conical kinematics which differs from the above by a coordinate transformation.
	
	For any $J\in \binom{\lbrack n\rbrack}{k}^{nf}$, define the linear functional
	\begin{eqnarray}\label{eq: gamma def}
		\gamma'_J =\sum_{i=1}^{k-1}\alpha_{i,\lbrack j_i,j_{i+1}-i-1\rbrack},
	\end{eqnarray}
	Thus, Conjecture \ref{conjecture: conical kinematics amp} has the following equivalent formulation in the second set of variables.
	\begin{conjecture}
		For any $\alpha\in \mathbb{R}^{(k-1)(n-k)}$, solve the equations $\eta_J(s) = \gamma'_J$ for the coordinate functions $s$ on $\mathcal{K}(k,n)$.  Then the scattering equations possess a unique solution, and we have 
		$$m^{(k)}(\mathbb{I}_n,\mathbb{I}_n)\big\vert_{\eta_J = \gamma'_J} = \prod_{i=1}^{k-1}\left( \frac{\sum_{j=1}^{n-k}\alpha_{i,j}}{\prod_{j=1}^{n-k}\alpha_{i,j}}\right).$$
	\end{conjecture}

	\section{Discussions}
	
	In this paper we are seeing that the planar kinematic point and its surrounding polytopal neighborhood $\Pi_{k,n}$ on the integer lattice, lie at the core of a seemingly vast network of connections and novel structures, between the CHY formulation of the biadjoint cubic scalar theory and its CEGM generalization, and mirror symmetry, tropical geometry, integrable systems, enumerative combinatorics and lattice polytopes.
	
	We found that the set of critical points of the planar kinematics potential function can be identified with equivalence classes of certain critical points of the superpotential defined in   \cite{eguchi1997gravitational} and more recently in \cite{marsh2020b}.  We have initiated the study of a highly structured lattice polytope $\Pi_{k,n}$ that surrounds the planar kinematic point, where $\Pi_{2,n}$ is a degeneration of the associahedron; we have studied its dual polytope, the rank-graded root polytope $\mathcal{R}_{k,n}$, where $\mathcal{R}_{2,n}$ is a projection of the type $A_{n-2}$ root polytope.  We have checked to nontrivial values of $(k,n)$ that the volume of $\mathcal{R}_{k,n}$ is the multi-dimensional Catalan number modulo a normalization constant, 
	$$\text{Vol}(\mathcal{R}_{k,n}) = \frac{C^{(k)}_{n-k}}{((k-1)(n-k-1))!}.$$
	
	We have reformulated the CEGM generalization of the cubic scalar theory as a single integral, with a tropical integrand, which can be evaluated explicitly, as opposed to its formulation as a sum over generalized Feynman diagrams (GFD). This integral can be interpreted as a Laplace transform of the whole ${\rm Trop}+G(k,n)$ where the dual space is the space of kinematic invariants. This is a significant advance, since now (in theory) one can evaluate $m^{(k)}_n$ without the computational task of constructing the arrays of Feynman diagrams with compatible metrics. Instead, the internal lengths of the diagrams in the GFD have been``glued" together to form a $(k-1)(n-k-1)$-dimensional space. Our main motivation for inotrducing this object was to use it as a tool to explore possible combinatorial structures hidden in (resummations of) GFDs which account for the appearance of k-dimensional Catalan numbers as the value of CEGM amplitudes at the PK point. However, it is clear that even for general kinematics the integral formula could have many applications. For example in the study of soft theorems \cite{Sepulveda:2019vrz,Cachazo:2019ble,Abhishek:2020xfy}. We leave the study of this fascinating object for future research.
	
	We have given a suggestive combinatorial interpretation for the domains of linearity for the tropicalized (exponentiated) planar kinematics potential function $\mathcal{S}_{2,n}$ which, dually, provides a combinatorial interpretation for the vertices of the planar kinematics associahedron, that is, as the facets of the root polytope $\mathcal{R}_{2,n}$.  It would be very interesting to extend this construction to $k\ge 3$.  This interpretation suggests the possibility that planar kinematics could be extended from a single point to a some bigger subset of the kinematic space; in fact we have seen exactly this in Corollary \ref{prop: PK point}, which provides a minimal polytopal neighborhood $\Pi_{k,n}$ of the planar kinematics point on the integer lattice.
	
	We have constructed rank-graded root polytopes $\mathcal{R}_{k,n}$; it is reasonable to expect that these possess a regular, unimodular triangulation into $\cat{k}{n-k}$ simplices; that will prove that the  relative volume conjecture for $\mathcal{R}_{k,n}$.  This suggests the natural possibility of an amplitude wherein the Feynman diagrams, that is maximal collections of compatible poles, are in bijection with the simplices in the unimodular triangulation of $\mathcal{R}_{k,n}$.  For this and other combinatorial structures associated to the polytopes $\mathcal{R}_{k,n}$ and $\Pi_{k,n}$, see \cite{CE2020}.
	
	%%%%%%%%%%%%%%%%%%%%%%%%
	
	Let us conclude some forward-looking and speculative issues raised through our work.
	
	In \cite{Early:2019zyi}, certain arrangements of the blade $((1,2,\ldots, n))$ on the vertices of hypersimplices $\Delta_{k,n}$ were shown to induce certain coarsest matroidal subdivisions and to provide a combinatorial criterion for compatibility of positroidal multi-splits; integrating integer translations $v$ of the blade $((1,2,\ldots, n))$ gives rise to a family of tropical polynomials $\rho_v$ on an integer lattice which satisfy a continuous analog of the octahedron recurrence (see Proposition \ref{prop: octahedron recurrence}); the vector of heights over the vertices $e_J\in \Delta_{k,n}$ defines a functional on the kinematic space $\mathcal{K}(k,n)$ and induces a planar basis element (\cite{Early:2019eun}), as used in this paper to characterize the PK point and a minimal polytopal neighborhood of it via an embedding $\Pi_{k,n}\hookrightarrow \mathcal{K}(k,n)$.
	
	On the other hand, in \cite{Borges:2019csl}, see also  \cite{Cachazo:2019xjx} and \cite{Guevara:2020lek}, Generalized Feynman Diagrams were introduced and used as a powerful combinatorial device for the calculation of maximal collections of compatible poles of $m^{(k)}(\mathbb{I}_n,\mathbb{I}_n)$.  
	
	Now one could consider blade arrangements on more general classes of lattice polytopes as well and it is reasonable to expect that corresponding notions of Generalized Feynman Diagrams would provide the characterization of maximal collections of compatible poles.  Put another way, what are the admissible subdivisions of polytopes, such as the type $A_3$ root \textit{solid} in Figure \ref{fig:incompatiblesplits}, that would be detected by Generalized Feynman Diagrams to induce maximal collections of compatible poles of an amplitude?
	\begin{figure}[h!]
		\centering
		\includegraphics[width=0.6\linewidth]{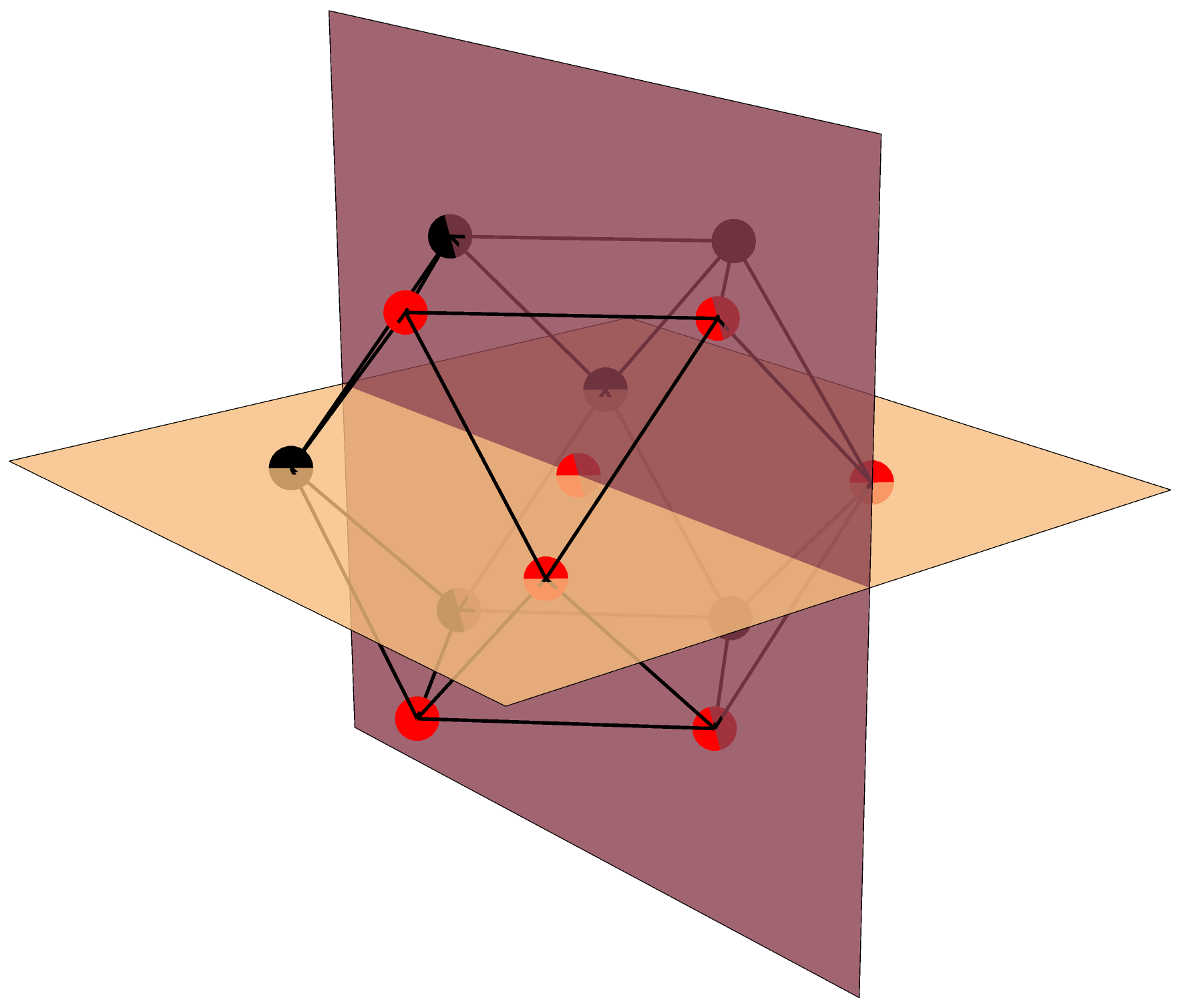}
		\caption{Incompatible 2-splits of the type $A_3$ root \textit{solid}.  The root \textit{polytope} $\hat{\mathcal{R}}_{2,5}$ is the convex hull of the red vertices.}
		\label{fig:incompatiblesplits}
	\end{figure}
	In Figure \ref{fig:incompatiblesplits}, the root \textit{polytope} $\hat{\mathcal{R}}_{2,5}$ is the convex hull of the red vertices, and by projecting along the diameter through the red vertex $e_1-e_4$ and its opposite black vertex $-(e_1-e_4)$ we obtain the root polytope $\mathcal{R}_{2,5}$.   There, the pair of orthogonal hyperplanes correspond locally to the blade arrangement $\{((1,2,3,4))_{e_{13}}, ((1,2,3,4))_{e_{24}}\}$.  The root polytopes $\hat{\mathcal{R}}_{k,n}$ and $\mathcal{R}_{k,n}$ introduced in Section \ref{sec:PK polytope} provide a rank-graded higher analog of the type $A$ root polytope, in the sense that $\mathcal{R}_{k,n}$ contains all lower-order polytopes $\mathcal{R}_{k-j,n-j}$ as faces, see Proposition \ref{prop: graded embeddings root polytopes}.  A first-order question is to investigate what $\mathcal{R}_{k,n}$, and $\Pi_{k,n}$ and its deformations, can tell us about the generalized biadjoint scalar $m^{(k)}(\mathbb{I}_n,\mathbb{I}_n)$.
	
	For a perhaps more ambitious question, let us first take a slight detour to discuss the classical Steinmann relations: we shall approach Figure \ref{fig:incompatiblesplits} from a new direction.

	The classical Steinmann relations, which provide a condition on what are the possible discontinuities of generalized retarded Green's functions, lie at the foundation of axiomatic QFT and an incarnation has appeared more recently in the study of planar $\mathcal{N}=4$ SYM \cite{Caron-Huot:2016owq}.  Very recently they have appeared in the context of combinatorial species, see \cite{Norledge:2020veg}.

	The classical Steinmann relations (see for instance the review article \cite{Streater:1975vw} and the references therein) can be formulated as a constraint on which are the pairwise compatible hyperplanes that cut through the center of the type A root solid, formed as the convex hull of all roots $e_i-e_j$ for $i\not=j$, as follows: for proper nonempty subsets $I,J$ of $\lbrack n\rbrack = \{1,\ldots, n\}$, a pair of hyperplanes
	$$x_I = 0,\ \ x_J = 0$$
	corresponds to an admissible discontinuity of a Green's function provided that at least one of the following four intersections is empty:
	\begin{eqnarray}\label{eq:Steinmann}
		I\cap J,\ \ I\cap J^c,\ \ I^c\cap J,\ \ I^c\cap J^c.
	\end{eqnarray}
	It is natural to view these conditions as a pointwise analog of the condition for pairwise compatibility relation of (2-split) matroid subdivisions of the hypersimplex $\Delta_{2,n}$.
	
	Returning to the root solid, the Steinman relations in Equation \eqref{eq:Steinmann} constitute a compatibility criterion for the common refinement of a pair of 2-splits of the type $A_{n-2}$ root solid, of which the root polytope $\hat{\mathcal{R}}_{2,n}$ is a subpolytope.
	
	Therefore one would like to ask for analogous compatibility relations for splits of the root polytopes $\hat{\mathcal{R}}_{k,n}$ and $\mathcal{R}_{k,n}$ and related root solids.
	
	Techniques developed to work with positroidal subdivisions for Generalized Feynman Diagrams \cite{Borges:2019csl} and \cite{Cachazo:2019xjx}, matroidal blade arrangements in \cite{Early:2019eun,Early:2020hap}, and their polymatroidal generalizations defined here, could provide an approach.
	
	\section*{Acknowledgements}
	
	We would like to thank W. Norledge, J. Scott, J. Tevelev for useful discussions. This research was supported in part by a grant from the Gluskin Sheff/Onex Freeman Dyson Chair in Theoretical Physics and by Perimeter Institute. Research at Perimeter Institute is supported in part by the Government of Canada through the Department of Innovation, Science and Economic Development Canada and by the Province of Ontario through the Ministry of Colleges and Universities.

	\appendix
	\section{Global Schwinger Formula for $m^{(k)}_n$}\label{sec: Global Schwinger}
	In this Appendix, we give details for the Global Schwinger Parametrization for CEGM generalized amplitudes $m^{(k)}_n$, introduced in Section \ref{sec:tropGrass Eval}.
	
	Here we introduce a particular\footnote{For related constructions, see \cite{speyer2005tropical} and the references therein.} positive parametrization of
	$$X^0(k,n) = \left\{g\in G(k,n): \prod_{J}\Delta_J(g)\not=0 \right\} \slash (\mathbb{C}^\ast)^n.$$
	For each $(a,b) \in \lbrack 1,k-1\rbrack \times \lbrack 1,n-k\rbrack$, define a polynomial
	$$M_{a,b} = \sum_{1\le j_a\le j_{a+1}\le \cdots \le j_{k-1}\le b}\left(\prod_{t=a}^{k-1}x_{t,j_t} \right).$$
	% $$M_{a,b} = \sum_{1\le j_1\le j_2\le \cdots \le j_{k-1}\le b}\left(\prod_{t=a}^{k-1}x_{t,j_t} \right).$$
	We fix the first $k$ columns of a matrix $M(x)$ to be the identity and the column k+1 to be monomials $M_{i,1} = x_{i,1}x_{i+1,1}\cdots x_{k-1,1}$ which are usually set to 1 to give the standard projective frame.
	$$M(x) = \begin{bmatrix}
		1 & 0 & \cdots  & &  0  & M_{1,1} & \cdots  & M_{1,n-k} \\
		0 & 1 &  &  & & M_{2,1} &  & M_{2,n-k} \\
		\vdots &  & \ddots &  & \vdots & \vdots & \vdots &  & \vdots  \\
		&  &  & 1 & 0  & M_{k-1,1} &  & M_{k-1,n-k} \\
		0 & 0 & \cdots & & 1 & 1 & 1 & \cdots  & 1
	\end{bmatrix}$$
	For example, when $(k,n) = (3,6)$ we find that $M(x)$ equals
	\begin{small}
		$$
		\begin{bmatrix}
			1 & 0 & 0 & x_{1,1} x_{2,1} & x_{1,2} x_{2,2}+x_{1,1} \left(x_{2,1}+x_{2,2}\right) & x_{1,3} x_{2,3}+x_{1,2} \left(x_{2,2}+x_{2,3}\right)+x_{1,1} \left(x_{2,1}+x_{2,2}+x_{2,3}\right) \\
			0 & 1 & 0 & x_{2,1} & x_{2,1}+x_{2,2} & x_{2,1}+x_{2,2}+x_{2,3} \\
			0 & 0 & 1 & 1 & 1 & 1 \\
		\end{bmatrix}$$
	\end{small}

	Denote by $\mathcal{F}^{(k)}_n$ the \textit{tropical potential}, 
	$$\mathcal{F}^{(k)}_n(y) = \sum_{J}P_J(y)\mathfrak{s}_J,$$
	where $P_J(y)$ are the (tropicalized) Plucker coordinates, evaluated on the above parametrization.  We usually assume that a chart has been chosen, taking (for example) $y_{i,1} = 0$ for $i=1,\ldots, k-1$.
	
	By the discussion which follows, assuming generic kinematics, then this can be viewed as a (dual) positive tropical Plucker vector (modulo lineality).
	
	For each pair $(L,\{i,j\})$ with $L \in \binom{\lbrack n\rbrack}{k-2}$ and $\{i,j\} \in \binom{\lbrack n\rbrack \setminus L}{2}^{nf}$.  Let 
	$$w^{(L)}_{ij} = \frac{\Delta_{L \cup \{i',j\}}\Delta_{L \cup \{i,j'\}}}{\Delta_{L \cup \{i,j\}}\Delta_{L \cup \{i',j'\}}},$$
	where $i'$ and $j'$ are the cyclic successors in $\lbrack n\rbrack \setminus L$ to $i$ and $j$, respectively, with respect to cyclic order inherited from $(1,2,\ldots, n)$.
	
	Define
	$$\text{Trop}(w^{(L)}_{ij}) = P_{L\cup \{i',j\}} + P_{L\cup \{i,j'\}} - P_{L\cup \{i,j\}}-P_{L\cup \{i',j'\}}.$$
	The positivity property in Lemma follows directly from the (quadratic) positive tropical Plucker relations.
	\begin{lem}\label{lem: positive trop cross ratios}
		Each planar cross-ratio $w^{(L)}_{ij}$ maps (the torus orbit of) the nonnegative Grassmannian into the interval $\lbrack 0,1\rbrack$.  For the tropicalization, one has
		$$\text{Trop}(w^{(L)}_{ij})(y) >0$$
		for all $y \in \mathbb{R}^{(k-1)\times(n-k-1)}$.
	\end{lem}
	
	% 		The positive tropical Plucker relations imply that all $\text{Trop}(w^{(L)}_{ij})$ are nonnegative for all $y \in \mathbb{R}^{(k-1)\times(n-k)}$, hence the integral converges.  
	% 	The proof of the last part of Proposition \ref{prop: Global Schwinger convergence} uses a key result about the $\alpha'$ expansion of the stringy integral, \cite[Claim 1]{Arkani-Hamed:2019mrd}.
	For each $J \in \binom{\lbrack n\rbrack}{k}$, let us write $\mathfrak{s}_J = \sum_{\{i,j\} \in \binom{J}{2}}s^{(J \setminus \{i,j\})}_{ij}$, where the new variables $s^{(J \setminus \{i,j\})}_{ij}$ are assumed to satisfy the relations
	$$\sum_{\ell \in \lbrack n\rbrack \setminus(J \cup \{i\}) }s^{(J \setminus \{i,j\})}_{i\ell} = 0.$$
	\begin{prop}\label{prop: Global Schwinger convergence}
		We have
		\begin{eqnarray}\label{eq:tropPot}
			\mathcal{F}^{(k)}_n(y) & = & \sum_{J \in \binom{\lbrack n\rbrack}{k}}P_J(y)\sum_{\{i,j\} \in \binom{J}{2}}s^{(J \setminus \{i,j\})}_{ij},\\
			& = & \sum_{L \in \binom{\lbrack n\rbrack}{k-2}}\sum_{\{i,j\}\in \binom{\lbrack n\rbrack \setminus L}{2}^{nf}}\text{Trop}(w^{(L)}_{ij})\eta^{(L)}_{ij},
		\end{eqnarray}
		where
		\begin{eqnarray}\label{eq: localized planar basis}
			\eta^{(L)}_{ij} & = & \sum_{\{a,b\} \subset L^c\cap \lbrack i+1,j\rbrack}s^{(L)}_{ab}.
		\end{eqnarray}
		If all of these are positive, $\eta^{(L)}_{ij}>0$, then the integral 
		$$m^{(k)}_n = \int_{\mathbb{R}^{(k-1)\times (n-k-1)}}\exp\left(-\mathcal{F}^{(k)}_n\right)dy$$
		converges; the value coincides with the CEGM scattering equations formula for $m^{(k)}_n$.  
	\end{prop}
	
	% 	\begin{prop}
		%         Equation \eqref{eq:tropPot} can be rewritten as 
		% 		$$\mathcal{F}^{(k)}_n(y) = \sum_{L \in \binom{\lbrack n\rbrack}{k-2}}\sum_{\{i,j\}\in \binom{\lbrack n\rbrack \setminus L}{2}^{nf}}\text{Trop}(w^{(L)}_{ij})\eta^{(L)}_{ij}.$$
		% 		If $\eta^{(L)}_{ij}>0$ are all positive, then the integral 
		% 		$$m^{(k)}_n = \int_{\mathbb{R}^{(k-1)\times (n-k-1)}}\exp\left(-\mathcal{F}^{(k)}_n\right)dy$$
		% 		converges.  Moreover, it is equal to the value obtained for $m^{(k)}_n$ from the CEGM formula.
		% 	\end{prop}
	
	\begin{proof}
		The verification of the expression for the $\eta^{(L)}_{i,j}$'s in Equation \eqref{eq: localized planar basis} is straightforward; indeed, the nontrivial step is to reorganize Equation \eqref{eq:tropPot} around the $\eta^{(L)}_{ij}$, one for each subset\footnote{More geometrically, one subset for each second hypersimplicial face $\Delta_{2,n-(k-2)}$ of the form $x_{\ell_1} = x_{\ell_2} = \cdots  = x_{\ell_{k-2}} = 1$ of the hypersimplex $\Delta_{k,n}$} $L\in \binom{\lbrack n\rbrack}{k-2}$.
		
		As per Lemma \ref{lem: positive trop cross ratios}, the positive tropical Plucker relations imply that all $\text{Trop}(w^{(L)}_{ij})$ are nonnegative for all $y \in \mathbb{R}^{(k-1)\times(n-k-1)}$, hence, provided that $\eta^{(L)}_{ij}$ are all positive, then the integral converges.
		
		% 		The final part is a consequence of Claim 1 of \cite{Arkani-Hamed:2019mrd}, that is, that the volume of the polar dual to the Newton polytope is equal (modulo to a possible sign) to the leading order term in the $\alpha'$ expansion of the stringy integral, which in turn coincides with the CEGM formulation of $m^{(k)}_n$
	\end{proof}
	
	\renewcommand{\thefigure}{\thesection.\arabic{figure}}
	\renewcommand{\thetable}{\thesection.\arabic{table}}
	\appendix

	\newpage
	
	\bibliographystyle{JHEP}
	\bibliography{references}
	
\end{document}